\documentclass[12pt]{article} 
\usepackage[sectionbib]{natbib}
\usepackage{array,epsfig,fancyheadings,rotating}
\usepackage[]{hyperref}  
\usepackage{xr-hyper}

\usepackage{amsmath}
\usepackage{amssymb}
\usepackage{amsfonts}
\usepackage{multirow}
\usepackage{amsthm}
\usepackage{graphicx}
\usepackage{enumitem}
\usepackage{cleveref}

\usepackage{custom_tex}
  
\usepackage{geometry} 
\newtheorem{theorem}{Theorem}
\newtheorem{lemma}[theorem]{Lemma}
\newtheorem{corollary}[theorem]{Corollary}
\newtheorem{proposition}[theorem]{Proposition}
\theoremstyle{definition}
\newtheorem{definition}{Definition}

\newcommand{\lh}{\hat{\ell}}
\newcommand{\ssf}{\textsf{F}}

\newcommand{\mP}{\mathbb{P}}
\newcommand{\E}{\mathbb{E}}

\begin{document}


\renewcommand{\baselinestretch}{2}

\markboth{\hfill{\footnotesize\rm JEHA YANG AND IAIN JOHNSTONE} \hfill}
{\hfill {\footnotesize\rm EDGEWORTH CORRECTION IN SPIKED PCA} \hfill}

\renewcommand{\thefootnote}{\fnsymbol{footnote}}
$\ $\par


\fontsize{12}{14pt plus.8pt minus .6pt}\selectfont \vspace{0.8pc}
\centerline{\large\bf EDGEWORTH CORRECTION FOR THE LARGEST EIGENVALUE}
\centerline{\large\bf IN A SPIKED PCA MODEL 
}
\vspace{.4cm} \centerline{Jeha Yang and Iain M. Johnstone} \vspace{.4cm} \centerline{\it
Stanford University} \vspace{.55cm} \fontsize{9}{11.5pt plus.8pt minus
.6pt}\selectfont



\begin{quotation}
\noindent {\it Abstract:}
We study improved approximations to the distribution of the largest
eigenvalue $\hat{\ell}$ of the sample covariance matrix of $n$
zero-mean Gaussian observations in dimension $p+1$. 
We assume that one population principal component
has variance $\ell > 1$ and the remaining `noise' components have
common variance $1$.  In the high dimensional limit
$p/n \to \gamma > 0$, we begin study of Edgeworth corrections to the
limiting Gaussian distribution of $\hat{\ell}$ in the supercritical
case $\ell > 1 + \sqrt \gamma$.  The skewness correction involves a
quadratic polynomial as in classical settings, but the coefficients
reflect the high dimensional structure.  The methods involve Edgeworth
expansions for sums of independent non-identically distributed
variates obtained by conditioning on the sample noise eigenvalues, and
limiting bulk properties \textit{and} fluctuations of these noise
eigenvalues.
\par

\vspace{9pt}
\noindent {\it Key words and phrases:}
Spiked PCA model, Roy's statistic, Edgeworth expansion
\par
\end{quotation}\par

\def\thefigure{\arabic{figure}}
\def\thetable{\arabic{table}}

\renewcommand{\theequation}{\thesection.\arabic{equation}}

\fontsize{12}{14pt plus.8pt minus .6pt}\selectfont


\setcounter{section}{1} 
\setcounter{subsection}{1} 
\setcounter{equation}{0} 

\noindent {\bf 1. Introduction}
\label{sec:introduction}

Models for high dimensional data with low dimensional structure are
the focus of much current research. 
This paper considers one of the simplest such settings, the rank one
``spiked model'' with Gaussian data, in order to begin the study of
Edgeworth expansion approximations for high dimensional data.
Specifically, we work with the following simple model.

\textbf{Model (M).} \ Suppose that we observe 
$X = [x_1, \cdots, x_n]'$ where $x_1, \ldots, x_n$ are i.i.d from $N_{p+1}(0,\Sigma)$, and the
population covariance matrix $\Sigma = I + (\ell - 1) v v'$ for
some unit vector $v$. Suppose also that $p$ increases with $n$ so that
$\gamma_n = p/n \to \gamma \in (0,\infty)$ and that 
$\ell > 1 + \sqrt \gamma$. 

Thus, one population principal component has variance $\ell > 1$ and
the remaining $p$ have common variance $1$.

The \cite{baik2005phase} \textit{phase transition} is an important
phenomenon that appears in this high dimensional asymptotic regime. 
It concerns the largest eigenvalues in spiked models, which are of
primary interest in principal components analysis.
In the rank one special case, let $\hat{\ell}$ be the largest
eigenvalue of the sample covariance matrix
$S = n^{-1} X'X$.
Below the phase transition, $\ell < 1 + \sqrt \gamma$, and after a
centering and scaling that does not depend on $\ell$, asymptotically
$n^{2/3} \hat{\ell}$ has a Tracy-Widom distribution.
Above the phase transition, the `super-critical regime', the convergence
rate is $n^{1/2}$ and the limit Gaussian:
\begin{equation}
\label{eq:G-lim}
n^{1/2} [\hat{\ell} - \rho(\ell,\gamma_n)]/\sigma(\ell,\gamma_n) 
\stackrel{\mathcal{D}}{\to} N(0,1).
\end{equation}
The centering and scaling functions now depend on $\ell$:
\begin{equation}
\label{eq:rho-sigma}
\rho(\ell,\gamma) = \ell + \gamma \ell/(\ell -1), \qquad 
\sigma^2(\ell,\gamma) = 2 \ell^2 [1-\gamma/(\ell-1)^2].
\end{equation}
\citet*{baik2005phase} proved \eqref{eq:G-lim} for complex valued data using
structure specific to the complex case.
The real case was established using different methods by 
\citet{paul2007asymptotics},
under the additional assumption $\gamma_n - \gamma = o(n^{-1/2})$ and
with $\gamma_n$ in \eqref{eq:G-lim} replaced by $\gamma$. 
We will see below that \eqref{eq:G-lim} holds as stated without this
assumption. Consequently, we adopt the abbreviations
\begin{equation} \label{eq:abbrev-moms}
\rho_n = \rho(\ell,\gamma_n), \qquad
\sigma_n = \sigma(\ell,\gamma_n).
\end{equation}

\lhead[\footnotesize\thepage\fancyplain{}\leftmark]{}\rhead[]{\fancyplain{}\rightmark\footnotesize\thepage}

The quality of approximation in asymptotic normality results such as 
\eqref{eq:G-lim} is often studied using Edgeworth expansions, e.g. 
\citet{hall1992bootstrap}.
However, our high dimensional setting appears to lie beyond the
standard frameworks for Edgeworth expansions, such as for example the
use of smooth functions of a \textit{fixed} dimensional vector of
means of independent random variables, as in \citet[Sec. 2.4]{hall1992bootstrap}.


\setcounter{section}{2} 
\setcounter{subsection}{1} 
\setcounter{equation}{0} 

\noindent {\bf 2. Main Result}
\label{sec:main-result}

Our main result is a skewness correction for the normal approximation
\eqref{eq:G-lim} to the largest eigenvalue statistic. 
The simplest version of the result may be stated as follows.
As usual $\Phi$ and $\phi$ denote the standard Gaussian cumulative and
density respectively.

\begin{theorem} \label{th:main-first}
	Adopt \textbf{Model (M)}, and let $\hat{\ell}$ be the largest
	eigenvalue of $S = n^{-1} \sum_{i=1}^n x_i x_i'$, and
	let $R_n = n^{1/2} (\hat{\ell}-\rho_n)/\sigma_n$, where the centering
	and scaling are defined in \eqref{eq:rho-sigma} and
	\eqref{eq:abbrev-moms}.
	Then we have a first order Edgeworth expansion
	\begin{equation}
	\label{eq:cor-exp}
	\mP(R_n \leq x) 
	= \Phi(x) + n^{-1/2} p_1(x) \phi(x) + o(n^{-1/2}),
	\end{equation}
	valid uniformly in $x$, and with
	\begin{equation} \label{eq:p1x}
	p_1(x) = \sqrt 2 \left( \tfrac{1}{3}[(\ell-1)^3 + \gamma](1-x^2) -
	\tfrac{1}{2} \gamma \ell \right)((\ell-1)^2-\gamma)^{-3/2}.
	\end{equation}
\end{theorem}

We compare \eqref{eq:cor-exp} with the previously known expression for
dimension $p$ fixed in the next section.
The effects of high dimensionality are seen both in the coefficient of
the ``usual'' polynomial $1 - x^2$ as well as in the additional
constant term proportional to $\gamma \ell$.

We turn to formulating the version of Theorem \ref{th:main-first} that
we actually 
prove, and in the process sketch some elements of our approach in
order to give a first indication of the role of high dimensionality in
the Edgeworth correction. 
Building on the approach 
of \citet{paul2007asymptotics}, 
the $n \times (p+1)$ data matrix may be partitioned as
$X = [ \sqrt \ell Z_1, Z_2]$, with the `signal' in the first column
and the remaining $p$ columns containing pure noise:
i.i.d. standard normal variates. 
Now consider the eigen decomposition $n^{-1} Z_2 Z_2' = U \Lambda U'$
in which $U$ is $n \times n$ orthogonal and the diagonal
matrix $\Lambda$ 
contains the ordered nonzero eigenvalues 
$\lambda_1 \geq \cdots \geq \lambda_{n \wedge p}$ 
of $n^{-1} Z_2 Z_2'$, supplemented by
zeros in the case $n > p$. 
It is a special feature of white Gaussian noise that $(U, \Lambda)$
are mutually independent, with $U$ being uniformly (i.e. Haar)
distributed on its respective space. 
In view of this, if we set $z = U'Z_1$, it follows that the
eigenvalues of $S$ depend only on $z$ and $\Lambda$, and that 
\begin{equation} \label{z_Lambda}
z = U' Z_1 \sim N(0,I_n), \quad z \perp \Lambda.
\end{equation}

The vector $z$ provides enough independent randomness for Gaussian limit
behavior of $\hat{\ell}$, conditional on $\Lambda$. 
In particular, for a function $f$ on $[0, \infty)$, we define
\begin{equation} \label{S_n_def}
S_n(f) = n^{-1/2} \sum_{i=1}^n f(\lambda_i) (z_i^2-1).
\end{equation}

As $n$ grows, we may also use the bulk regularity properties of
$\Lambda$.  Thus the empirical distribution $F_n$ of the $p$ sample
eigenvalues of $n^{-1} Z_2' Z_2$ converges to the Marchenko-Pastur distribution
$F_\gamma$ supported on
$[a(\gamma),b(\gamma)]$ if $\gamma \leq 1$ and with an atom $(1-\gamma^{-1})$ at $0$
if $\gamma > 1$, where 
$$ a(\gamma) = (1-\sqrt{\gamma})^2, \quad b(\gamma) = (1+\sqrt{\gamma})^2.$$
The `companion' empirical distribution $\ssf_n$ of the $n$ eigenvalues 
$(\lambda_1, \ldots, \lambda_n)$ 
of $n^{-1} Z_2 Z_2'$ converges
to the companion MP law $\ssf_\gamma = (1-\gamma)I_{[0,\infty)} +
\gamma F_\gamma$. Integrals against $F$ indicating one of these types of distributions will be
written in the form
\begin{equation*}
F(f) = \int f(\lambda) F(d \lambda).
\end{equation*}

Paul's Schur complement argument, reviewed in the proof section below, leads
to an equation for the fluctuation of $\hat{\ell}$ about its centering
$\rho_n$:
\begin{equation}
\label{eq:onestep}
n^{1/2} (\hat{\ell}-\rho_n)
= \frac{S_n(g_n)}{\ssf_{\gamma_n}(g_n^2)} + O_p(n^{-1/2}),
\end{equation}
where $g_n(\lambda) = (\rho_n - \lambda)^{-1}$. 
From \eqref{F_g_n_2}, 
$\ssf_{\gamma_n}(g_n^2) = 2 \sigma_n^{-2}$. 
The sum $S_n(g_n)$ is asymptotically normal given $\Lambda$,
with asymptotic variance $\ssf_{\gamma}(g^2)$, for
example via the Lyapounov CLT, and completing this argument yields the
asymptotic normality result \eqref{eq:G-lim}. 

A more accurate version of \eqref{eq:onestep} is needed for a first
Edgeworth approximation. Indeed we later show that
\begin{equation*}
 n^{1/2} (\hat{\ell}-\rho_n)
= \frac{S_n(g_n)+n^{-1/2}G_n(g_n)}{\ssf_{\gamma_n}(g_n^2)
	+n^{-1/2}G_n(\tilde{g}_n) + O_p(n^{-1})},
\end{equation*}
where $\tilde{g}_n$ is defined later.
This expression involves the discrepancy between a trace and its
centering:
\begin{equation*}
G_n(f) = \sum_{i=1}^n f(\lambda_i) - n \int f(\lambda) \ssf_{\gamma_n}(d \lambda)
= n ( \ssf_n(f)  - \ssf_{\gamma_n}(f) ) = p(F_n(f) - F_{\gamma_n}(f)).
\end{equation*}
This centered linear statistic, though unnormalized, is $O_p(1)$, and indeed, according to the CLT of 
\citet{bai2004clt}, for suitable $f$ is asymptotically normal:
\begin{equation} \label{eq:BSlim}
G_n(f) \stackrel{\mathcal{D}}{\to} N(\mu(f), \sigma^2(f)).
\end{equation}

We use a first term Edgeworth approximation to the distribution of
$S_n(g_n)$ conditional on $\Lambda$,
using results for sums of independent non-identically distributed
variables described in \citet[Ch VI.]{petrov1975sums}.
This uses the conditional
cumulants of $S_n$ for
$j = 2, 3$, given by
\begin{equation*}
\frac{d^j}{dt^j} \log \E [e^{itS_n}|\Lambda] \vert_{t=0} 
= \kappa_j n^{-1} \sum_{i=1}^n g_n^j(\lambda_i),
\end{equation*}
where, in turn, $\kappa_j = 2^{j-1} (j-1)!$ are the cumulants of $z^2
- 1  \sim \chi_{(1)}^2 - 1$.
A deterministic asymptotic approximation to these conditional
cumulants is then given by
\begin{equation}
\label{eq:kappajn}
\kappa_{2,n} = 2 \ssf_{\gamma_n}(g_n^2), \qquad 
\kappa_{3,n} = 8 \ssf_{\gamma_n}(g_n^3).
\end{equation}
With these preparations we are ready for the main theorem. 

\begin{theorem}
	\label{th:main-second}
	With the assumptions of Theorem \ref{th:main-first}, we have 
	the Edgeworth expansion 
	\begin{equation}
	\label{eq:main-exp}
	\mP(R_n \leq x) 
	= \Phi(x) + n^{-1/2} p_{1,n}(x) \phi(x) + o(n^{-1/2}),
	\end{equation}
	valid uniformly in $x$, and with
	\begin{equation*}
	p_{1,n}(x) = \tfrac{1}{6} \kappa_{2,n}^{-3/2} \kappa_{3,n} (1-x^2) 
	- \kappa_{2,n}^{-1/2} \mu(g_n),
	\end{equation*}
	for $g_n(\lambda) = (\rho_n - \lambda)^{-1}$ and $\kappa_{j,n}$
	defined by \eqref{eq:kappajn}, and $\mu(\cdot)$ the asymptotic mean 
	in the Bai-Silverstein limit \eqref{eq:BSlim}.
\end{theorem}

The structure of $p_{1,n}(x)$ as an even quadratic polynomial is the
same as in the smooth function of means model 
\citep[Theorem 2.2]{hall1992bootstrap}. 
In our high dimensional setting, 
the first term in $p_{1,n}(x)$ reflects the Edgeworth approximation to
$S_n(g_n)$ conditional on $\Lambda$, while the second shows 
the effects of fluctations of $\Lambda$. 
From \eqref{F_g_n_2}, \eqref{F_g_n_3} and \eqref{mu_g_n}, we have more explicit evaluations
\begin{align*}
\kappa_{2,n} &= 2(1 - \ell^{-1})^2 ((\ell-1)^2 - \gamma_n)^{-1} =  4 \sigma_n^{-2}, \\
\kappa_{3,n}  &= 8(1- \ell^{-1})^3 ((\ell-1)^3 + \gamma_n) ((\ell-1)^2 - \gamma_n)^{-3}, \\
\mu(g_n) &=  \gamma_n (\ell-1) ((\ell-1)^2 - \gamma_n)^{-2},
\end{align*}
which lead to an explicit form of the first order
correction term 
\begin{equation*}
p_{1,n}(x) = \sqrt{2}
\left( \tfrac{1}{3}[(\ell-1)^3 + \gamma_n](1-x^2) -
\tfrac{1}{2} \gamma_n \ell \right)((\ell-1)^2-\gamma_n)^{-3/2}. 
\end{equation*}
Since the error term is $o(n^{-1/2})$ and $\gamma_n = \gamma + o(1)$,
we may replace $\gamma_n$ by $\gamma$ in the previous display and recover 
Theorem \ref{th:main-first}.

\bigskip 
\textbf{Remark.} \ 
To emphasize the advantage of using  $\gamma_n = p/n$ rather
than $\gamma$ in the
centering and scaling formulas,
note that if $\gamma_n = \gamma + a n^{-1/2}$,
then the limiting distribution of
\begin{equation*}
\check R_n = n^{1/2}[\hat{\ell}-\rho(\ell,\gamma)]/\sigma(\ell,\gamma)
\end{equation*}
has a non-zero mean $\alpha = \alpha(a,\ell,\gamma)$. 
The situation is yet more delicate for the skewness correction: if
$\gamma_n = \gamma + b n^{-1}$, then
\begin{equation*}
\mP(\check R_n \leq x) - \mP( R_n \leq x)
= n^{-1/2} (\beta_0 + \beta_1 x) \phi(x) + o(n^{-1/2})
\end{equation*}
for constants $\beta_1, \beta_0$ depending on $b,\ell, \gamma$.

\textbf{Remark.} A parallel result for rank one perturbations of the
Gaussian Orthogonal Ensemble  is available.
Consider a data matrix  $X = \theta e_1 e_1^{T} + Z$ where
$\theta > 1$ and $Z$ is $p \times p$ symmetric with
$Z_{ii} \sim N(0,2/p)$ and $Z_{ij} \sim N(0,1/p)$ for $i > j$, and
$p \to \infty$.  
The largest eigenvalue of $X$, denoted $\hat{\theta}$, converges a.s. to
$\rho = \theta + \theta^{-1}$, 
and with $\sigma =
\sqrt{2(1-\theta^{-2})}$,
the quantity $R_p = \sqrt{p}(\hat{\theta} - \rho) / \sigma$ is
asymptotically standard Gaussian
\citep[Theorem 5.1]{benaych2011fluctuations}.
As is well known, the empirical spectral distribution
of $Z^{[2:p, 2:p]}$ converges weakly to the \textit{semicircle
	law} $F_{sc}$ with density 
$\frac{1}{2\pi}\sqrt{4-x^2}$ on the interval $[-2, 2])$.
Our method,
along with CLT for linear spectral statistics $F_{sc}(f)$ of
\cite{bai2005on} leads to a first order Edgeworth correction for
$R_p$:
\[
p_{1}(x) = \frac{\sqrt{2}}{(\theta^2 - 1)^{3/2}} \left( \frac{1-x^2}{3}
- \frac{1}{2} \right), 
\]
which has a structure analogous to that of our main result.

\bigskip
\textbf{Comparison with fixed $p$.} \ 
In classical asymptotic theory, when $n \to \infty$ with $p$ fixed, 
asymptotically $\hat{\ell} \sim N(\ell, 2 \ell^2)$.
Introduce therefore $\check R_n = \sqrt n(\hat{\ell}-\ell)/(\sqrt 2 \ell)$.
When specialized to the skewness correction term, Theorem 2.1 of \citet{muirhead1975asymptotic}
reads
\begin{equation}
\label{eq:mc75}
\mP(\check R_n \leq x)
= \Phi(x) + n^{-1/2} \bigg(\frac{\sqrt 2}{3}(1-x^2) -
\frac{p}{\sqrt 2(\ell-1)}\bigg) \phi(x) + O(n^{-1}).
\end{equation}
Formally setting $\gamma = 0$ in \eqref{eq:p1x} of Theorem
\ref{th:main-first}, we get only the term $p_1(x) = (\sqrt 2/3)(1-x^2)$. 
To see that the two results are nevertheless consistent, write 
$\rho_n = \ell(1+b_n)$ and $\sigma_n = \sqrt 2 \ell c_n$ where $b_n = \gamma_n/(\ell-1)$ and $c_n = [1-\gamma_n/(\ell-1)^{2}]^{1/2} $, 
so that
\begin{equation*}
R_n = \sqrt{n} \frac{\hat{\ell}-\ell -b_n \ell}{\sqrt 2 \ell c_n}
= c_n^{-1}(\check R_n - d_n),
\end{equation*}
where $d_n = \sqrt{n/2} b_n = \sqrt{n/2} \gamma_n  /(\ell-1) = (2n)^{-1/2} p/(\ell-1)$ is
the second term in \eqref{eq:mc75}.
Applying \eqref{eq:mc75} at $\check x_n = c_n x + d_n$, we find
\begin{equation*}
\mP(R_n \leq x) 
= \mP(\check R_n \leq \check x_n)
= \Phi(\check x_n) + [n^{-1/2} \tfrac{\sqrt 2}{3} (1 - \check x_n^2)
- d_n] \phi(\check x_n) + O(n^{-1}).
\end{equation*}
Observe that $\Phi(\check x_n) - d_n \phi(\check x_n) = \Phi(c_n x) + O(d_n^2)$ with $d_n = O(n^{-1/2})$, 
and $c_n = [1-\gamma_n/(\ell-1)^{2}]^{1/2} = 1 + O(n^{-1})$.
Therefore, $\check x_n = x + O(n^{-1/2})$ and $ c_n x = x + O(n^{-1})$,
yielding
\begin{equation*}
\mP(R_n \leq x) 
= \Phi(x) + n^{-1/2} \tfrac{\sqrt 2}{3} (1-x^2) \phi(x) + O(n^{-1}),
\end{equation*}
and so we do recover agreement with $\gamma = 0$ in \eqref{eq:p1x}.

\bigskip
\textbf{Hermite polynomials and numerical comparisons.} \ 
It is helpful to view Edgeworth expansions in terms of Hermite
polynomials $H_n(x)$, defined by
$H_n(x) \phi(x) = (-d/dx)^n \phi(x)$. In particular,
$H_n(x) = 1, x, x^2-1$ and $x^3 - 3x$ for $n = 0,1,2$ and $3$. 
The Edgeworth approximation of Theorem \ref{th:main-second} then becomes
\begin{align*}
F_E & = \Phi - n^{-1/2}(\alpha_2 H_2 + \alpha_0) \phi
\end{align*}
with $h = \ell -1$ and
\begin{align*}
\alpha_2 & = \frac{\sqrt 2}{3}
\frac{h^3+\gamma_n}{(h^2-\gamma_n)^{3/2}}, \qquad 
\alpha_0 = \frac{1}{\sqrt 2} \frac{\gamma_n l}{(h^2-\gamma_n)^{3/2}}.
\end{align*}
Since $(d/dx) H_n(x) = - H_{n+1}(x)$, the Edgeworth corrected density
is given by
\begin{equation*}
f_E = \phi + n^{-1/2} (\alpha_2 H_3 + \alpha_0 H_1) \phi.
\end{equation*}
The relative error
\begin{equation*}
\frac{f_E-\phi}{\phi} = n^{-1/2} q, \qquad \qquad
q = \alpha_2 H_3 + \alpha_0 H_1,
\end{equation*}
is a cubic polynomial with positive leading coefficient.
It is easy to verify that the three roots, namely $0, \pm
(3-\alpha_0/\alpha_2)^{1/2}$ are real when $\ell > 1 + \sqrt{\gamma_n}$.
Hence the Edgeworth density approximation is necessarily negative for
$\hat{\ell}$ sufficiently small, and intersects the normal density
three times. 

We now show numerical examples in which the Edgeworth corrected
`density' provides a better approximation to the distribution of $R_n$
than does the standard normal.
The parameters
\begin{equation*}
n \in \{ 50, 100 \} ; \qquad  \gamma_n \in \{0.1, 1\} ; \qquad \ell \text{-factor}
:=  \ell / (1 + \sqrt{\gamma_n}) - 1  \in \{0.3, 0.5\}, 
\end{equation*}
are chosen so that 
$n$ is neither too small for asymptotics to be meaningful nor too
large to distinguish $f_E(x)$ and $\phi(x)$, 
$\gamma_n$ is close to either 0 or 1, 
and $\ell$ is moderately separated from the (finite version) critical
point $1+\sqrt{\gamma_n}$.

Figures \ref{figure_0.3} and \ref{figure_0.5}
\begin{figure}[tb] 
	\caption{Plots for $l$-factor $= 0.3$} \label{figure_0.3}
	\begin{minipage}[b]{0.49\linewidth}
		\centering
		\includegraphics[width = 80mm]{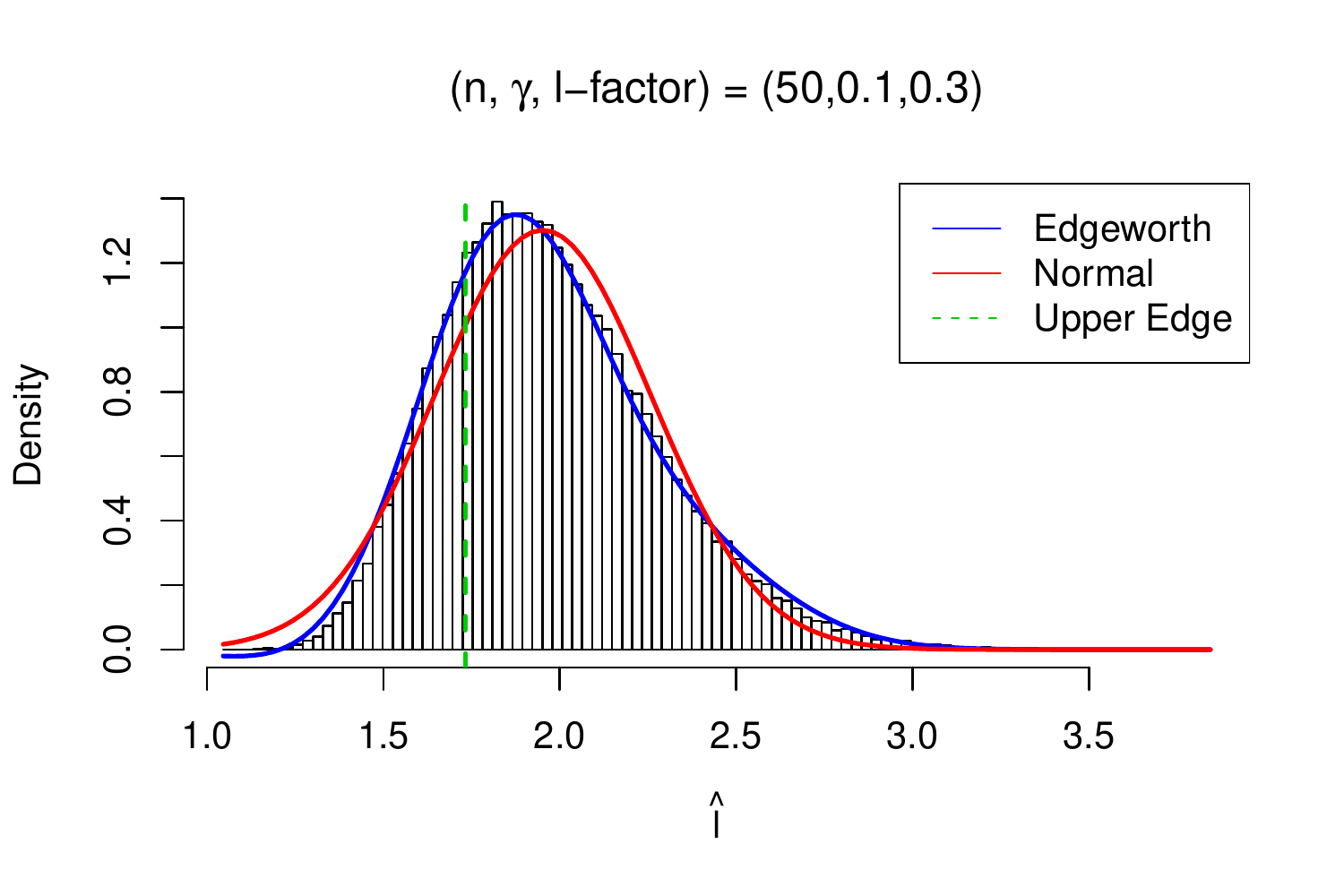}
	\end{minipage}
	\hfill	
	\begin{minipage}[b]{0.49\linewidth}
		\centering
		\includegraphics[width = 80mm]{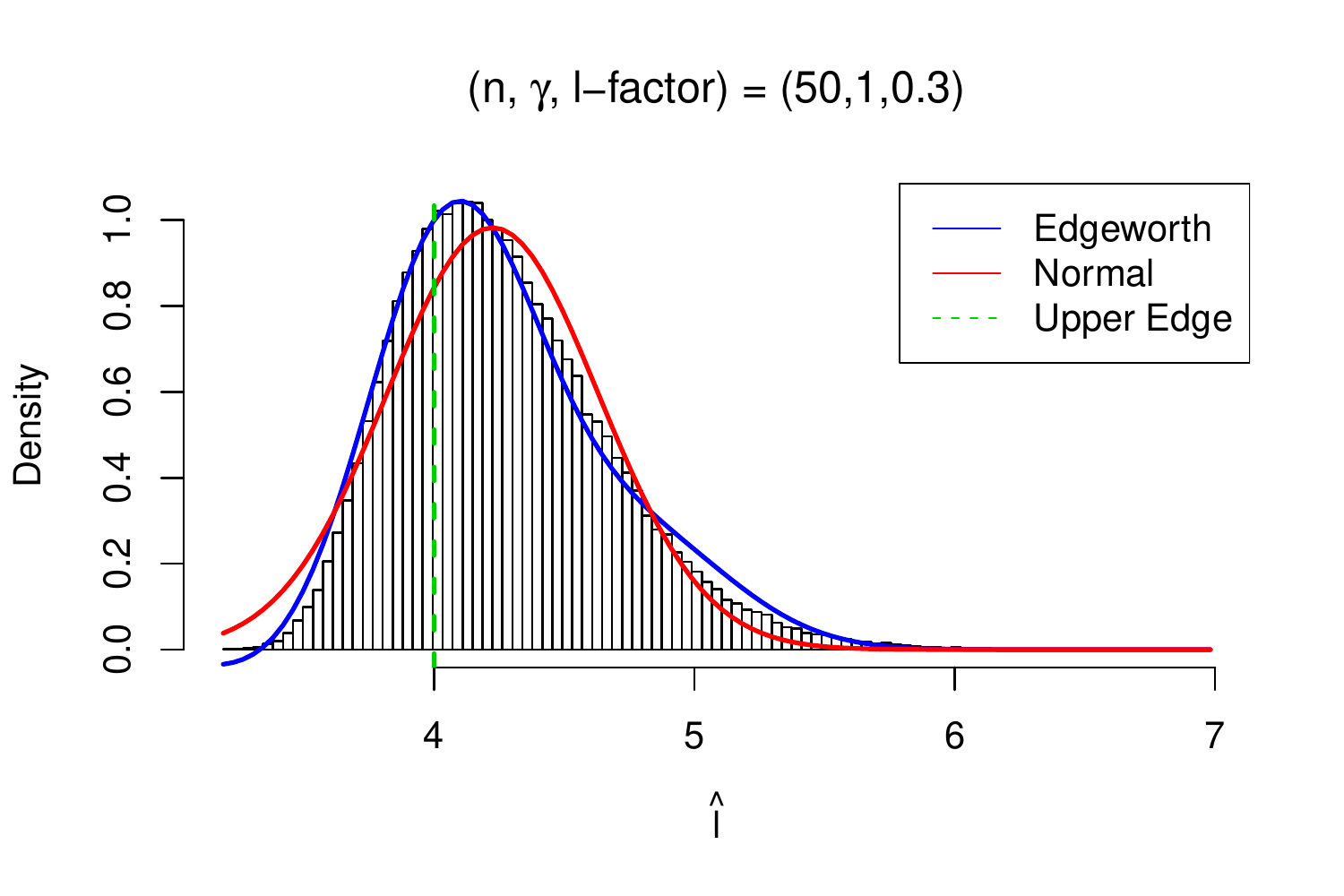}
	\end{minipage}
	\\
	\begin{minipage}[b]{0.49\linewidth}
		\centering
		\includegraphics[width = 80mm]{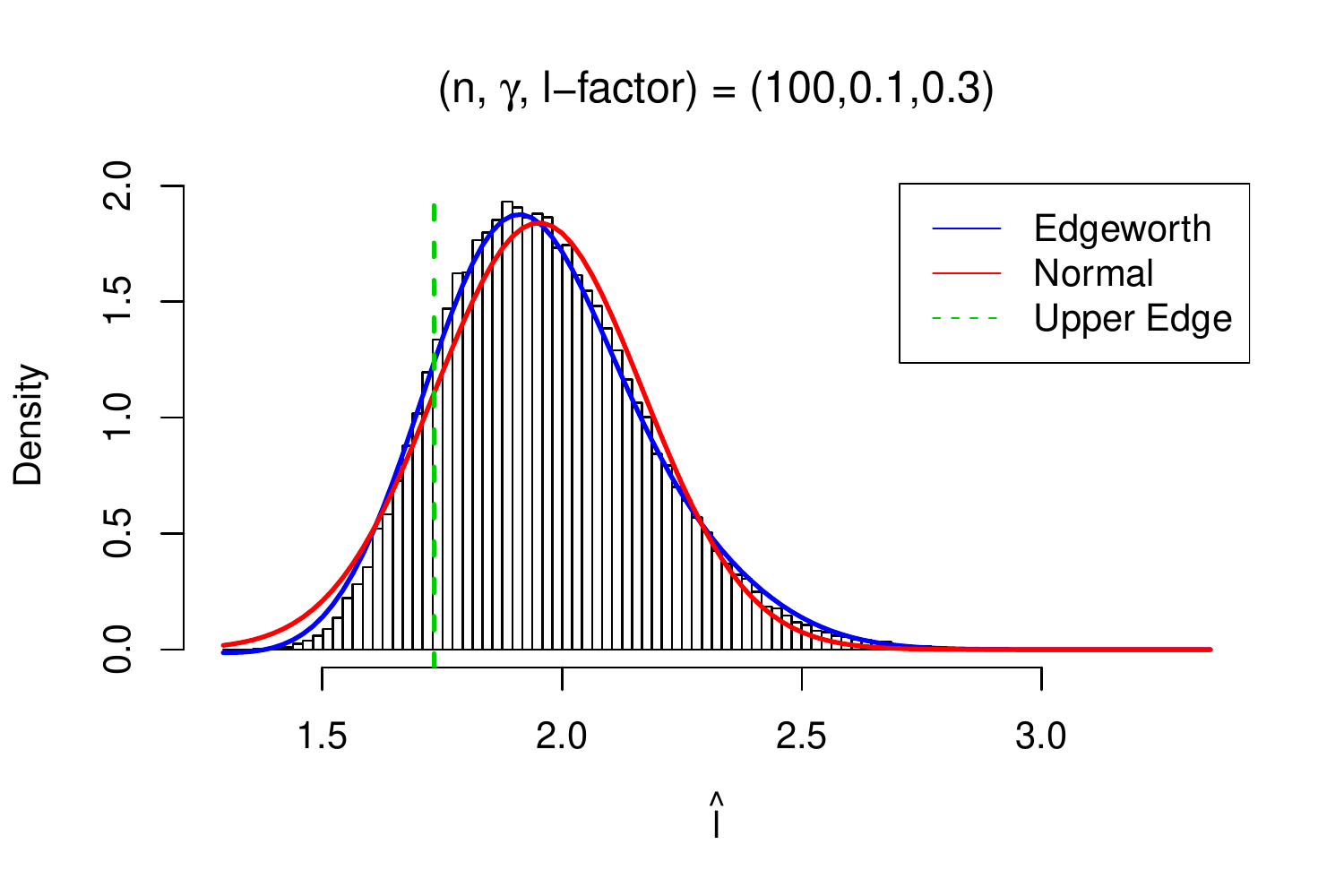}
	\end{minipage}
	\hfill	
	\begin{minipage}[b]{0.49\linewidth}
		\centering
		\includegraphics[width = 80mm]{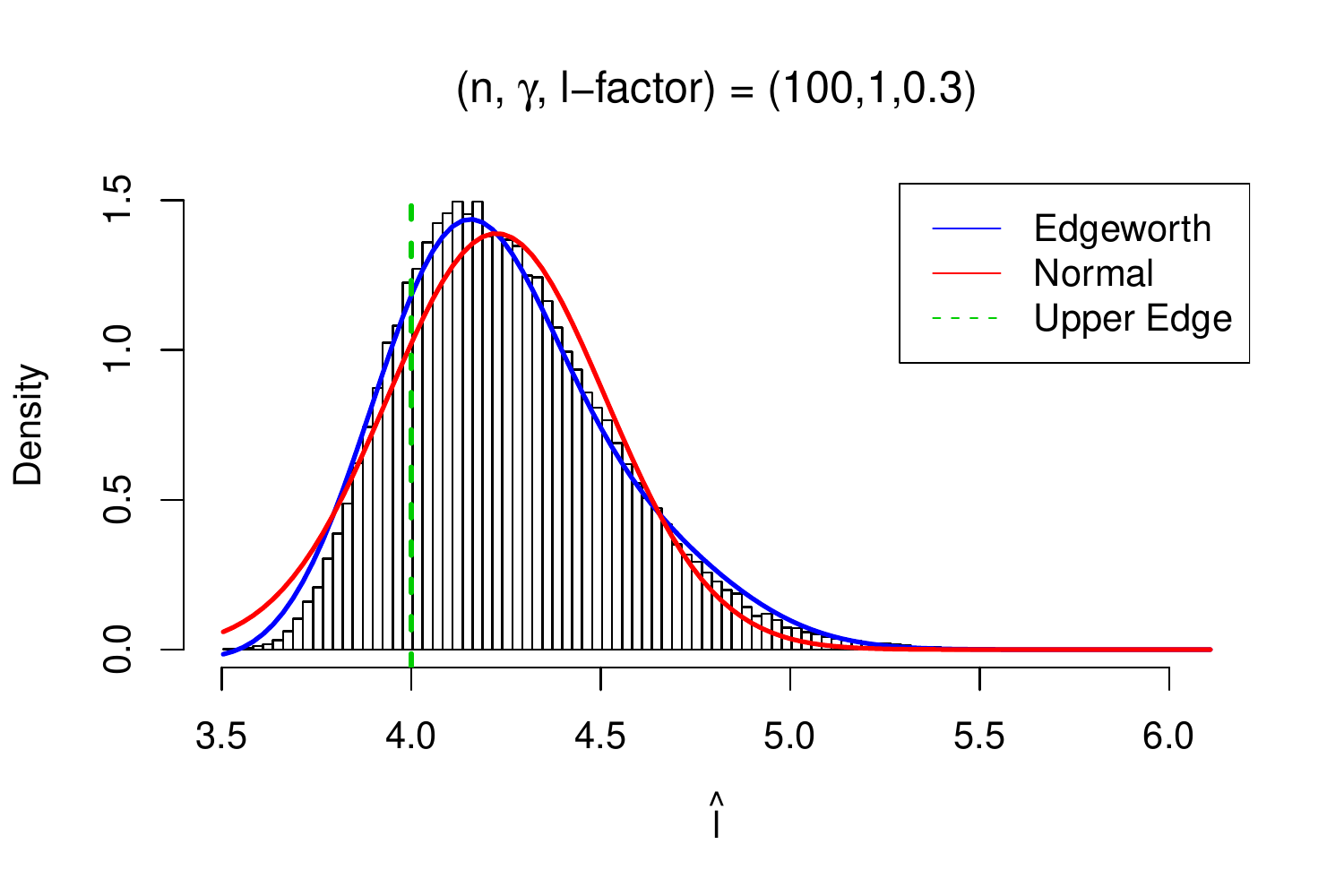}
	\end{minipage}
\end{figure}
\begin{figure}[htb] 
	\caption{Plots for $l$-factor $= 0.5$} \label{figure_0.5}
	\begin{minipage}[b]{0.49\linewidth}
		\centering
		\includegraphics[width = 80mm]{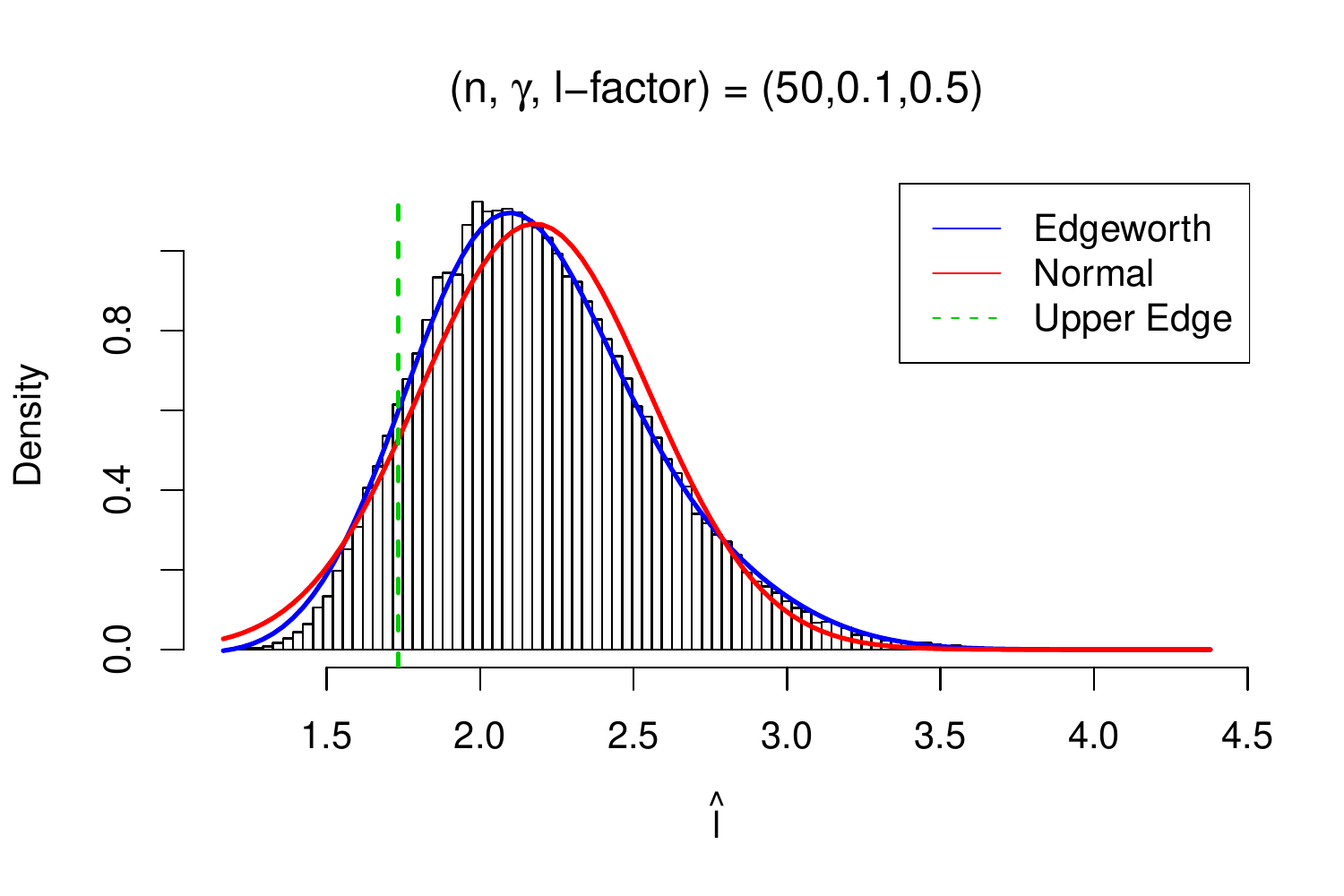}
	\end{minipage}
	\hfill	
	\begin{minipage}[b]{0.49\linewidth}
		\centering
		\includegraphics[width = 80mm]{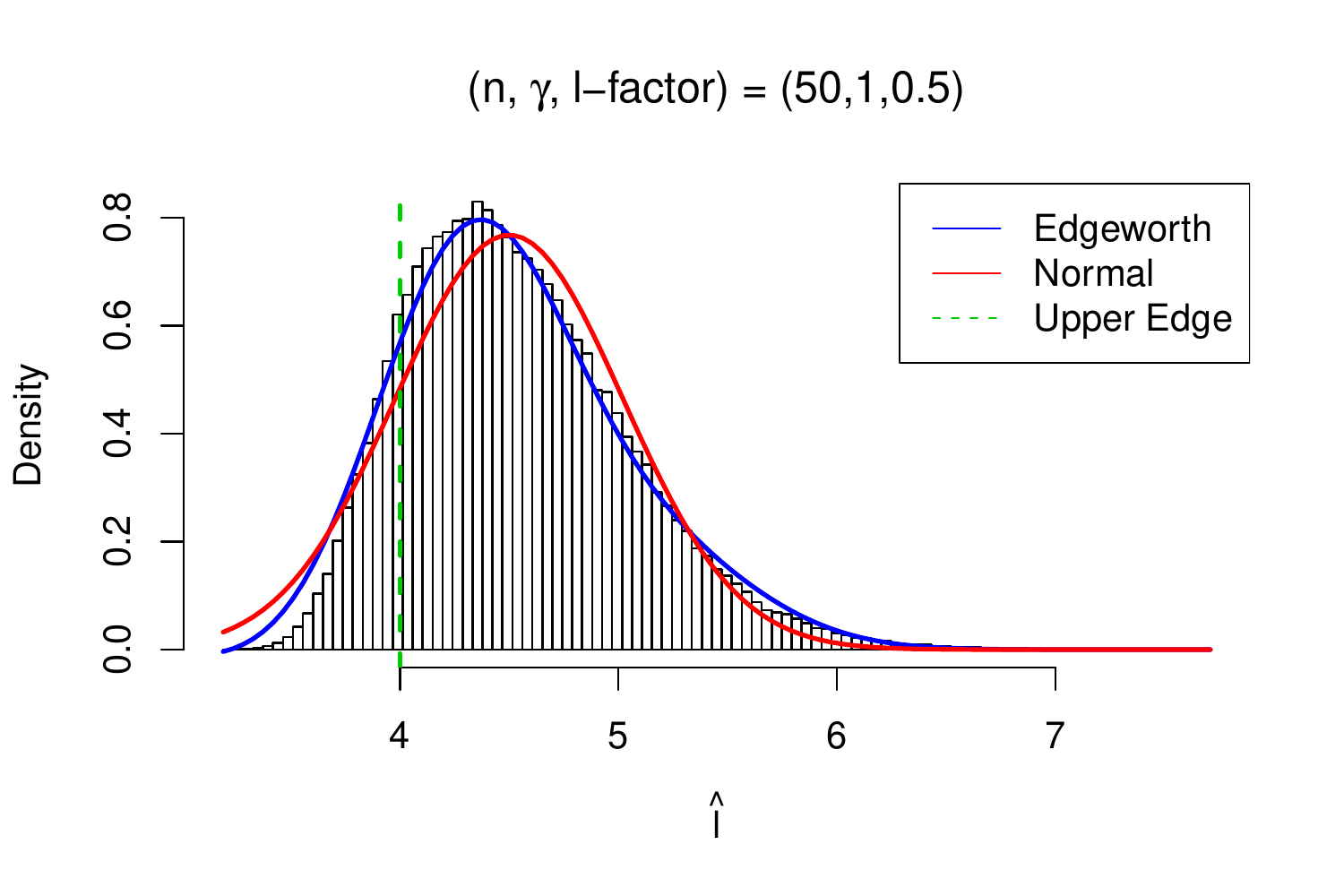}
	\end{minipage}
	\\
	\begin{minipage}[b]{0.49\linewidth}
		\centering
		\includegraphics[width = 80mm]{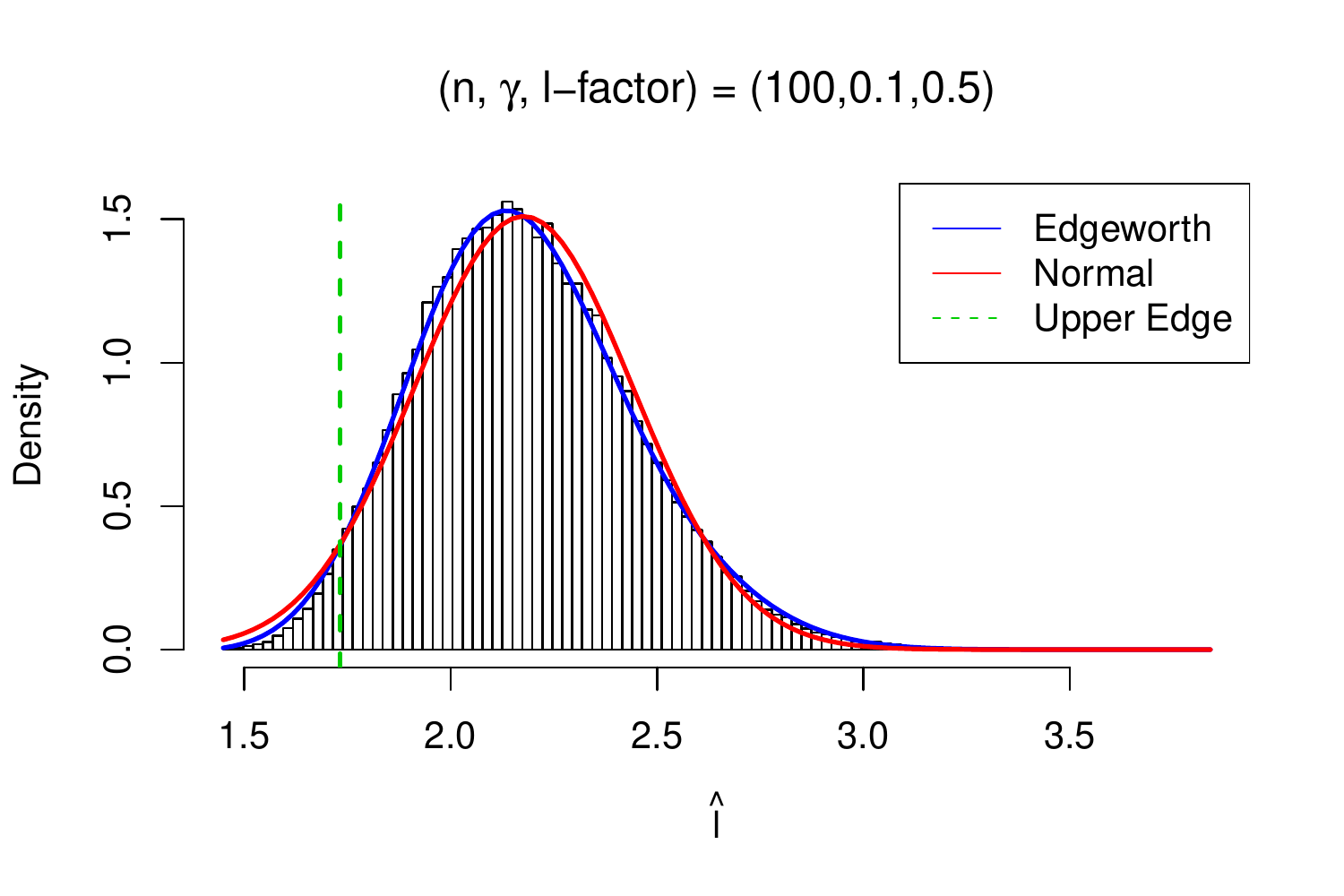}
	\end{minipage}
	\hfill	
	\begin{minipage}[b]{0.49\linewidth}
		\centering
		\includegraphics[width = 80mm]{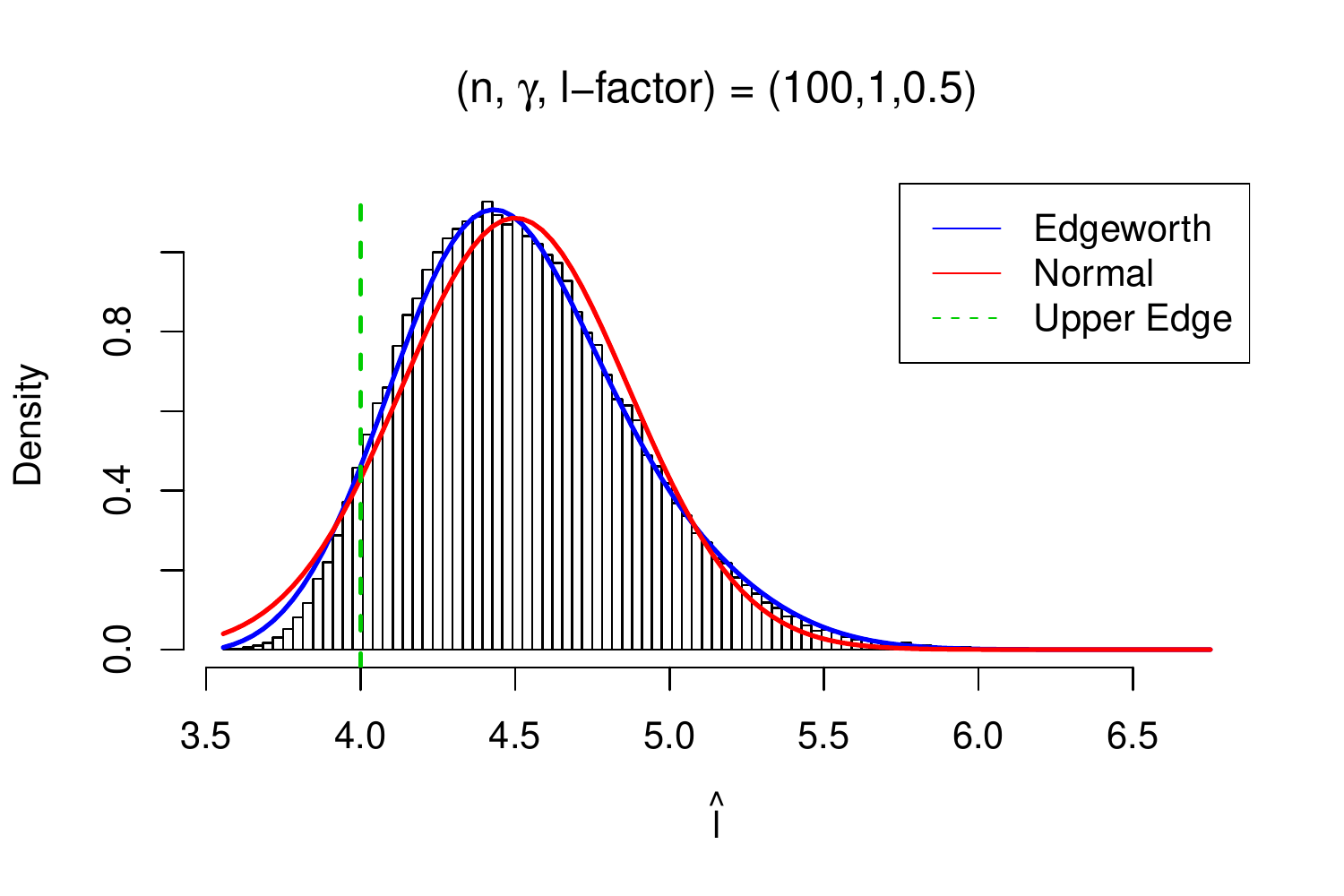}
	\end{minipage}
\end{figure}
in fact show the densities
$y \to \sqrt{n/\sigma_n} f_E(\sqrt{n/\sigma_n}(y-\rho_n))$ after
shifting and scaling to correspond to $\hat{\ell}$. 
Superimposed are the corresponding rescaled normal density as well as
histograms of $100,000$ simulated replicates of $\hat{\ell}$.
The green dashed lines show the upper bulk edge $(1 + \sqrt{\gamma_n})^2$ to emphasize that these settings for $\hat{\ell}$ are not
too far above the bulk.
In the cases shown, the Edgeworth correction provides a (right)
skewness correction that matches the simulated histograms reasonably
well, though unsurprisingly the small $n = 50$ and large $\gamma = 1$
case, has the least good match.

When $\ell$ is closer to the phase transition, so that the
$\ell$-factor is smaller, the skewness correction becomes
unsatisfactory due to the singularity in the denominator of $\alpha_2$
and $\alpha_0$ as $h$ approaches $\sqrt \gamma$. 
Empirically, we have found that the skewness
correction may be reasonable, with a single inflection point visible above
the mode, when
\begin{equation*}
\frac{1}{n}  (9/2) \alpha_2^2 = \frac{1}{n} \frac{(h^3+\gamma)^2}{(h^2-\gamma)^3}
\leq 0.2.
\end{equation*}


\setcounter{section}{3} 
\setcounter{subsection}{1} 
\setcounter{equation}{0} 

\noindent {\bf 3. Proof}

\noindent {\bf 3.1. Outline}

We start with deriving the useful expression of $R_n$ as introduced in the first section with more details.
Without loss of generality, we may assume that the population covariance matrix of the distribution of $x_1, \cdots, x_n$ is $\text{diag}(\ell, 1, \cdots, 1)$(by an appropriate rotation, not changing $S$). 
Then, we write $X = [ \sqrt{\ell} Z_1 \ \ Z_2]$ where $Z_1, Z_2$ are $n \times 1$, $n \times p$ with i.i.d. standard normal elements, respectively.
The eigenvalue equation $S \hat{v} = \lh \hat{v}$ becomes
\begin{equation*}
\begin{pmatrix}
\ell Z_1'Z_1  &  \sqrt{\ell} Z_1'Z_2 \\
\sqrt{\ell} Z_2' Z_1 & Z_2'Z_2
\end{pmatrix}
\begin{pmatrix}
\hat{v}_1 \\ \hat{v}_2
\end{pmatrix}
= n \lh   \begin{pmatrix}
\hat{v}_1 \\ \hat{v}_2
\end{pmatrix},
\end{equation*}
where $\hat{v}_1, \hat{v}_2$ are the first coordinate and the rest of $\hat{v}$, respectively.
As usual, we substitute the second equation into the first, then
cancel $\hat{v}_1$ to obtain
\begin{equation*}
n \lh 
= \ell Z_1' [I_n + Z_2 (n\lh I_p -Z_2' Z_2 )^{-1}Z_2'] Z_1 
= \ell Z_1' [ \lh ( \lh I_n - n^{-1}Z_2 Z_2')^{-1} ] Z_1 
= \ell z'[- \lh R(\lh)]z,
\end{equation*}
whenever $\det(n \lh I_p - Z_2' Z_2) \neq 0$, i.e. almost surely.
Note that the second equation is a particular case of the Woodbury formula, $z = U'Z_1$ where $U$ is from the eigendecomposition $n^{-1} Z_2 Z_2' = U \Lambda U'$ as introduced before, and the resolvent $R(x) = (\Lambda - xI_n)^{-1}$ is defined for 
$x \notin \{ \lambda_1, \ldots, \lambda_n \}$. 
Now using the resolvent identity $R(x) = R(y) + (x - y) R(x) R(y)$
for $x, y \notin \{\lambda_1, \cdots, \lambda_n\}$, we obtain
\begin{align*}
	n \lh &= \ell z'[ - \rho_n R(\rho_n) - (\lh - \rho_n)\Lambda R(\lh) R(\rho_n)]z,
\end{align*}
which can be rearranged into a key equation 
\begin{equation} \label{key_eq}
(\lh - \rho_n) (1 + \ell n^{-1} z' \Lambda R(\lh) R(\rho_n) z)
= \ell \rho_n (- n^{-1} z' R(\rho_n) z - \ell^{-1})
\end{equation}
whenever $\lh, \rho_n \notin \{\lambda_1, \cdots, \lambda_n\}$ i.e. almost surely ; we assume this from now on. 
To investigate \eqref{key_eq} further, we will make frequent use of the \textit{stochastic decomposition}
\begin{equation} \label{stoc_decomp}
n^{-1} \sum_{i=1}^n f(\lambda_i) z_i^2 
= \ssf_{\gamma_n}(f) + n^{-1/2} S_n(f) + n^{-1} G_n(f).
\end{equation}
where $\ssf_n(\cdot), S_n(\cdot)$ and $G_n(\cdot)$ are defined as above,
which are of order $O_p(1)$ as we will see in the proof section.
Noting that $ -R(\rho_n) = \text{diag} ( g_n(\lambda_1), \cdots,  g_n(\lambda_n))$ and $\ssf_{\gamma_n}( g_n) = \ell^{-1}$ \eqref{F_g_n}, we have
$- n^{-1} z' R(\rho_n) z
=\ell^{-1} + n^{-1/2} S_n(g_n) + n^{-1} G_n(g_n)
$ 
from \eqref{stoc_decomp}. 
Hence we can rewrite \eqref{key_eq} as
\begin{equation}\label{key_eq_2}
(\lh - \rho_n) (1 + \ell n^{-1} z' \Lambda R(\lh) R(\rho_n) z) = 
n^{-1/2} \ell \rho_n (S_n(g_n) + n^{-1/2} G_n(g_n)).
\end{equation}
Also, use the resolvent identity to write
\begin{equation} \label{key_eq_denom}
1 + \ell n^{-1} z' \Lambda R(\lh) R(\rho_n) z 
 = 1 + \ell n^{-1} z' \Lambda R^2(\rho_n) z - \ell \nu_n,
\end{equation}
where
\begin{equation} \label{nu_n}
\nu_n = -(\lh - \rho_n) n^{-1} z' \Lambda R(\lh) R^2(\rho_n) z
\end{equation}
will be $O_p(n^{-1/2})$ by \eqref{key_eq_2} and tail bounds. 
One can use \eqref{stoc_decomp} to write the leading term as
\begin{equation} \label{key_eq_denom_leading}
	1 + \ell n^{-1} z' \Lambda R^2(\rho_n) z 
	= \ell \rho_n \ssf_{\gamma_n}(g_n^2) + n^{-1/2}\ell S_n(m_1 g_n^2) + n^{-1}\ell G_n(m_1 g_n^2)
\end{equation}
where $m_k(\lambda) := \lambda^k, k\in\NN $ are monomials, 
since $1 + \ell \ssf_{\gamma_n}(m_1 g_n^2) - \ell \rho_n \ssf_{\gamma_n}(g_n^2) = 1 - \ell \ssf_{\gamma_n} (g_n) = 0$ again by \eqref{F_g_n}.
This allows us to rewrite \eqref{key_eq_2} as 
\[
n^{1/2}(\lh - \rho_n) = \frac{S_n(g_n) + O_p(n^{-1/2})}{\ssf_{\gamma_n}(g_n^2) + O_p(n^{-1/2})} 
\]
which establishes \eqref{eq:onestep}. 
To expand $\nu_n$ further, we insert \eqref{eq:onestep} into \eqref{nu_n}, yielding
\begin{equation} \label{nu_n_2}
  \begin{split}
\nu_n & = n^{-1/2}( S_n(g_n) / \ssf_{\gamma_n}(g_n^2) + O_p(n^{-1/2}))
(\ssf_{\gamma_n}(m_1 g_n^3) + O_p(n^{-1/2}) ) \\
      & = n^{-1/2} r_n S_n(g_n) + O_p(n^{-1}),
  \end{split}
\end{equation}
where 
\begin{equation}
r_n 
= \ell \rho_n \ssf_{\gamma_n}(m_1 g_n^3) / (1 + \ell \ssf_{\gamma_n}(m_1 g_n^2)) 
= \ssf_{\gamma_n}(m_1 g_n^3) / \ssf_{\gamma_n}( g_n^2).  \label{r_n}
\end{equation}
Putting \eqref{key_eq_denom_leading}, \eqref{nu_n_2} and $\ssf_{\gamma_n}(g_n^2) = 2 \sigma_n^{-2}$ \eqref{F_g_n_2} into \eqref{key_eq_denom} gives
\begin{align} \label{key_eq_denom_2}
1 + \ell n^{-1} z' \Lambda R(\lh) R(\rho_n) z 
& = \ell ( 2 \rho_n \sigma_n^{-2} + n^{-1/2} S_n( m_1 g_n^2 - r_n g_n ) + \delta_n)
\end{align}
where 
\begin{equation}
	\delta_n = n^{-1} G_n(m_1 g_n^2) - (\nu_n - n^{-1/2} r_n S_n( g_n))
	\label{delta_n_nu_n}
\end{equation}
is $O_p(n^{-1})$ ignorable ; a rigorous proof of this fact is postponed to the delta method section.

All in all, combining \eqref{key_eq_2} and \eqref{key_eq_denom_2}, we obtain the master equation 
\begin{equation} \label{R_n_exp}
n^{1/2} (\lh - \rho_n)
= \frac{ \rho_n ( S_n(g_n) + n^{-1/2} G_n(g_n) )}{ 2 \rho_n \sigma_n^{-2}  - n^{-1/2} S_n( g_n h_n) + \delta_n },
\quad \text{with} \quad 
h_n = r_n - m_1 g_n.
\end{equation}

Now we are ready to see the outline of the main proof.
For notational convenience, let 
$ \eta(\ell, \gamma) := \rho(\ell, \gamma) - b(\gamma) = (\ell-1)^{-1} (\ell-1-\sqrt{\gamma})^2 > 0$.

\begin{enumerate} [label=\textbf{{Step}{ \arabic*}}]

\item \label{step1} From \textit{tail bounds}, show that for any fixed $\delta \in (0, \min (1, \eta(\ell, \gamma)/4, \gamma / 2) )$, the event
\begin{equation} \label{E_0n}
E_{0,n} = \{ \lambda_1 + \delta < \min\{\rho(\ell, \gamma), \rho_n, \lh\}, \ssf_n(m_2) - \ssf_n(m_1)^2 > \gamma^2 / 8\}
\end{equation}
 is of probability  $1 - O(\exp(-c n^{1/2}))$ for a positive $c$ depending only on $\gamma, \ell, \delta$.
Therefore, $\PP{R_n \leq x} - \PP{E_{0,n} \cap \{ R_n \leq x\} } = O(\exp(-c n^{1/2}))$ uniformly in $x \in \RR$, i.e. it suffices to do the analysis on $E_{0,n}$. 
Then, for notational convenience, let $\EE[n]{X} := \EE{I(E_{0,n})X}$ and $\PP[n]{E} := \PP{E_{0,n} \cap E}$ for any random variable $X$ and event $E$.

\item \label{step2} Using \eqref{R_n_exp}, \textit{linearize} the event $ \{ R_n \leq x \}$ as
\begin{align} 
\{R_n \leq x \} 
&= \{\rho_n  ( S_n(g_n) + n^{-1/2} G_n(g_n) ) 
\leq  (2 \rho_n \sigma_n^{-2}  - n^{-1/2} S_n(g_n h_n) + \delta_n) \sigma_n x\} \nonumber \\
&= \{M_n -  \delta_n x_n \leq 2 \sigma_n^{-1} x \} \label{linearization}
\end{align}
where $x_n = \rho_n^{-1} \sigma_n x$ and $M_n$, the main linearized statistic, is defined as
\begin{align} \label{M_n}
M_n &:= S_n( ( 1  + n^{-1/2}  x_n h_n) g_n ) + n^{-1/2}  G_n(g_n).
\end{align}

\item \label{step3} Use the \textit{Edgeworth expansion for sums of independent random variables} to expand $\PP{M_n \leq 2 \sigma_n^{-1} x \mid \Lambda }$ on $E_{0,n}$ up to the accuracy of $o(n^{-1/2})$ uniformly in $x\in\RR$.
Then take its expectation over $\Lambda$ to obtain the corresponding expansion of $\PP[n]{M_n \leq 2 \sigma_n^{-1} x}$.

\item \label{step4} Apply the \textit{delta method for Edgeworth expansion} to obtain
\begin{align} \label{apply_delta_method}
\PP[n]{ R_n \leq x } 
&= \PP[n]{ M_n \leq 2 \sigma_n^{-1} x } + o(n^{-1/2})
\end{align}
uniformly on $x \in \RR$.

\end{enumerate}

\noindent {\bf 3.2. Bai-Silverstein CLT}

As a core component of our analysis, a particular case of the CLT for linear spectral statistics from \cite{bai2004clt} is introduced.

\begin{theorem} \label{B_S_CLT}
Suppose that $Z_n:= [z_1 \cdots z_n]$ with $z_1, \cdots, z_n \overset{i.i.d.}{\sim} N(0, I_p)$ and $\gamma_n := p/n \rightarrow \gamma \in \RR^+$ as $n \rightarrow \infty$.
As defined above, let $F_n(x)$ and $F_{\gamma_n}(x)$ be the empirical spectral distribution of $Z_n Z_n^{t}/p$ and the Marchenko-Pastur distribution with the parameter $\gamma_n$ respectively, and 
$G_n(x) := p(F_n(x) - F_{\gamma_n}(x))$.
Then, for any real function $f$ analytic on an open interval containing $I(\gamma) := [I(\gamma \in (0,1))a(\gamma), b(\gamma)]$,
\begin{align*}
G_n(f) \overset{d}{\rightarrow} N(\mu(f), \sigma^2(f)),
\end{align*}
where $\mu(f)$ and $\sigma^2(f)$ are finite values determined by $\{ f(x) \mid x\in I(\gamma) \}$.
In particular, $\mu(f)$ is given by ((5.13) of \cite{bai2004clt})
\[
\mu(f) = \frac{f(a(\gamma)) + f(b(\gamma))}{4} - \frac{1}{2\pi}\int_{a(\gamma)}^{b(\gamma)} \frac{f(x)}{\sqrt{4\gamma - (x - 1 -\gamma)^2}} dx. 
\]
\end{theorem}

It is clear that Bai-Silverstein CLT is applicable for $g(\lambda) :=  (\rho(\ell, \gamma) - \lambda)^{-1}$, because $ \rho(\ell, \gamma) - b(\gamma) = \eta(\ell, \gamma) > 0$. 

\noindent {\bf 3.3. Tail bounds}

We introduce tail bounds in this section in order to establish \ref{step1}, i.e. to separate $ \lambda_1 $ from $\min\{\rho(\ell, \gamma), \rho_n, \lh\}$, and $\ssf_n(m_2)$ from $\ssf_n(m_1)^2$, with overwhelming probability.
All proofs are postponed to the section \textbf{S2}.

We start with $\lambda_1 $  and $\min\{\rho(\ell, \gamma), \rho_n \}$.
Note that $ \min\{\rho(\ell, \gamma), \rho_n \} - b(\gamma) > \delta$ for some positive $\delta$ and all large enough $n$, so the following proposition is sufficient.
\begin{proposition} 
[Proposition 1 of \cite{paul2007asymptotics}] \label{paulProposition}
For each $\delta \in (0, b(\gamma)/2)$,
the event $ E_{1,n} := \{  \lambda_1 > b(\gamma) + \delta \}$ satisfies
$$ \mathbb{P}(E_{1,n}) \leq \exp (- 3 n \delta^2 / ( 64b(\gamma) )) $$
for all $n > n_{\delta}$, where $n_{\delta} \in \NN$ is determined by $\delta$ and $\{\gamma_n \}_{n\in \NN}$.
\end{proposition}
Now assume $\delta \in (0, \min ( \eta(\ell, \gamma) / 3, b(\gamma)/2 ) )$ and choose $n_0(\delta) \in \NN$ such that $|\rho_n - \rho(\ell, \gamma)| < \delta$ for all $n > n_0(\delta)$.
Then, on $E_{1, n}^c$
\[
\lambda_1 + \delta \leq b(\gamma)  + 2 \delta < \rho(\ell, \gamma) - \delta < \min\{\rho(\ell, \gamma), \rho_n\}
\]
for all $n > n_0(\delta)$, as desired.

The next 2 propositions are to restrict $| \lh - \rho_n |$ on $E_{1, n}^c$, resulting in separation between $ \lambda_1 $ and  $\min\{\rho(\ell, \gamma), \rho_n, \lh \}$.
Observe that 
\[
\lh = \sup_{v \in \mathbb{S}^{p-1}} \|S v\|_2 
> \sup_{w \in \mathbb{S}^{p-2}} \|S^{[2:(p+1),2:(p+1)]} w\|_2 = \lambda_1
\]
whenever $\hat{v}_1 \neq 0$, hence 
$ z' \Lambda R(\lh) R(\rho_n) z \geq 0$
almost surely on $E_{1,n}^c$.
This leads to
\begin{equation} \label{WT_l_hat_ineq}
|l \rho_n (S_n(g_n) + n^{-1/2} G_n(g_n))| = (1+l n^{-1} z' \Lambda R(\lh) R(\rho_n) z) | n^{1/2} (\lh - \rho_n)| \geq |n^{1/2} (\lh - \rho_n)|
\end{equation}
almost surely on $E_{1,n}^c$, from \eqref{key_eq_2}.
Therefore, it suffices to find tail bounds for $S_n(g_n)$ and $G_n(g_n)$ on $E_{1, n}^c$.
We introduce propositions for more general settings, which will be necessary in the delta method for Edgeworth expansion section.

\begin{proposition} \label{TailBound_Indep}
For $M > 0$ and a function $f$ absolutely bounded by $U_f$ on $[0, b(\gamma) + \delta]$,
$E_{2,n}(f, M) := \{|S_n(f)| > M\}$ satisfies $$\mathbb{P}(E_{1,n}^c \cap E_{2,n}(f, M)) \leq 15\exp (-M/ U_f).$$
\end{proposition}

\begin{proposition} \label{TailBound_LSS}
For functions $\{f_n\}_{n\in\NN}$ such that (i) $f_n(x^2), n\in\NN$ share a Lipschitz constant $L$ on $[0, (b(\gamma) + \delta)^{1/2} ]$ (as functions of $x$) and (ii) $\{G_n(f_n)\}_{n\in\NN}$ is uniformly tight, then
\begin{align}
M(\{f_n\}_{n\in\NN}) := \sup_{n\in\NN} | \EE{ G_n( \textsf{f}_n) } | \quad \text{with} \quad \textsf{f}_n(\lambda): = f_n( (\lambda \vee 0) \wedge (b(\gamma) + \delta) )  
\end{align}
is finite. Furthermore, for $M > 2 M(\{f_n\}_{n\in\NN})$,
$E_{3,n}(f_n, M) := \{|G_n(f_n)| > M\}$ satisfies
$$ \mathbb{P} (E_{1,n}^c \cap E_{3,n}(f_n, M) ) \leq 2 \exp (-M^2 / (8 L^2))).$$
\end{proposition}

\Cref{TailBound_Indep} immediately follows from the Markov inequality for moment generating functions, while \Cref{TailBound_LSS} is mainly based on Corollary 1.8 (b) of \cite{guionnet2000concentration}.

To apply \Cref{TailBound_LSS}, assumptions (i) and (ii) need to be established for all sufficiently large $n$ ; 
(i) is true when $f'_n$ exists and is uniformly bounded on $[0, b(\gamma) + \delta]$ because $( f_n(x^2))' = 2x f'_n(x^2)$.
For (ii), the following lemma provides a sufficient condition.
\begin{lemma} \label{G_n_lemma}
In the setting of \Cref{B_S_CLT}, suppose there is an open neighborhood $\Omega \subset \mathbb{C}$ of $I(\gamma)$
such that (i) $\{ f_n \}_{n\in\NN}$ is analytic and locally bounded in $\Omega$
and  (ii) $f_n \to f$ pointwise on $I(\gamma)$.  Then
\begin{equation*}
G_n(f_n) - G_n(f) \stackrel{p}{\to} 0
\end{equation*}
as $n \to \infty$. In particular, $G_n(f_n)$ has the same limiting Gaussian distribution as
$G_n(f)$.  
\end{lemma}
The proof relies on and adapts parts of the proof of \cite{bai2004clt} Theorem 1.1, along with the Vitali-Porter and Weierstrass
theorems(e.g. \citet[Ch.~1.4, 2.4]{schiff2013normal}).
This lemma is sufficient for the uniform tightness required for (ii) of \Cref{TailBound_LSS}, because of Slutsky's theorem and Prohorov's Theorem(e.g. \cite{van2000asymptotic} Theorem 2.4).
Consequently, we obtain the following corollary.

\begin{corollary} \label{Cor_TailBound_LSS}
For functions $\{f_n\}_{n\in\NN}$, assume that for $n' \in \NN$
(i) $\{f'_n\}_{n > n'}$ is uniformly bounded by $L'$ on $[0, b(\gamma) + \delta]$, 
(ii) $\{ f_n \}_{n > n'}$ is analytic and locally bounded in an open neighborhood $\Omega \subset \mathbb{C}$ of $[a(\gamma), (1 + \sqrt{\gamma})^2]$
and  (iii) $f_n \to f$ pointwise on $[a(\gamma), (1 + \sqrt{\gamma})^2]$.
Then $G_n(f_n) \overset{d}{\to} N(\mu(f), \sigma^2(f))$ and 
$$\mathbb{P} (E_{1,n}^c \cap E_{3,n}(f_n, M)) \leq 2 \exp (-M^2 / (32 (b(\gamma) + \delta) L'^2)) $$
 for $M > 2 M(\{f_n\}_{n > n'})$ and all $n > n'$.
\end{corollary}

Now it is easy to see that $\{g_n\}_{n > n'}$ satisfies sufficient conditions for \Cref{TailBound_Indep} and \Cref{Cor_TailBound_LSS} for $U_f = \delta^{-1} $, $n' = n_0(\delta)$ and $L' = \delta^{-2}$, from $ | g_n(\lambda) | \leq (\rho_n - b(\gamma) - \delta)^{-1} < \delta^{-1}$ for all $\lambda \in  [0, b(\gamma) + \delta]$ and $n > n_0(\delta)$.
Hence, \eqref{WT_l_hat_ineq} gives
\begin{corollary} \label{Cor_l_hat}
	For any $\delta \in (0, \min ( \eta(\ell, \gamma) / 3, b(\gamma)/2 ) )$ and $M > 0$ , 
	$$\mathbb{P} ( E^c_{1,n} \cap \{n^{1/2} | \lh - \rho_n | > M \} ) = O(\exp ( - c(\gamma, \ell, \delta) M))$$
	for a constant $c(\gamma, \ell,\delta)$ depending only on $\gamma, \ell, \delta$.
\end{corollary}

Finally, we verify \ref{step1} as follows :
let $\delta \in (0, \min ( \eta(\ell, \gamma) / 3, \gamma/2 ) )$ and take $\epsilon > 0$ such that 
$\epsilon^2 + 3\epsilon < \gamma^2 / 8 $.
Then, if $\max(|G_n(m_2)|, |G_n(m_1)|) \leq n \epsilon $ for $ n > n_0(\delta)$,
\begin{align*}
\ssf_n(m_2) - \ssf_n(m_1)^2 
&\geq \ssf_{\gamma_n}(m_2) -  \epsilon - (\ssf_{\gamma_n}(m_1)+\epsilon)^2 
= \gamma_n^2 - (\epsilon^2 + 3\epsilon) 
> (\gamma - \delta)^2 - \gamma^2 / 8
> \gamma^2 / 8
\end{align*}
since $\ssf_{\gamma_n}(m_1) = 1, \ssf_{\gamma_n}(m_2) = 1 + \gamma_n^2$ from \citet[Proposition.~2.13]{yao2015sample}, and 
$\delta > |\rho_n - \rho(\ell, \gamma)| = \ell |\gamma_n - \gamma| / (\ell-1) \geq |\gamma_n - \gamma|$.
Therefore,
$
E_{1,n}^c \cap \{|\lh - \rho_n| \leq \delta \} \cap E_{3,n}^c(m_1, n \epsilon) \cap E_{3,n}^c(m_2, n \epsilon) \subset E_{0,n}
$
from \eqref{E_0n}, i.e. \ref{step1} is established by \Cref{paulProposition}, \Cref{TailBound_LSS} and \Cref{Cor_l_hat}.

Last but not least, we have the following corollary for moments for the future use, from \Cref{Cor_TailBound_LSS} and Theorem 2.20 of \cite{van2000asymptotic}.
\begin{corollary} \label{Cor_limiting_moment}
	 For functions $\{f_n\}_{n\in\NN}$ and $f$ satisfying the conditions for \Cref{Cor_TailBound_LSS} and any sequence of measurable $E_n$ such that $E_n \subset E_{1,n}^c$ and $\lim_{n\rightarrow\infty}\PP{E_n} = 1$,
	\[
	\lim_{n\rightarrow\infty}\EE{I(E_n) (G_n(f_n))^k} = \tau_k(f), \forall k\in\NN,
	\]
	where $\tau_k(f)$ denotes the $k^{\text{th}}$ moment of $N(\mu(f), \sigma^2(f))$.
	In particular, since $\{g_n\}_{n\in\NN}, g$ and $\{E_{0,n}\}_{n\in\NN}$ satisfy these sufficient conditions, $ \lim_{n\rightarrow\infty}\EE[n]{ (G_n(g_n))^k} = \tau_k(g)$ holds.
\end{corollary}

\noindent {\bf 3.4. Edgeworth expansion for sums of independent random variables}

A heuristic conversion between characteristic function and Edgeworth expansion is described in \citet[pg.~48]{hall1992bootstrap}.
Justification for the conversion is the main subject of Chapter VI of
\cite{petrov1975sums}, and leads to his Theorem 7, which we state in
modified form in Theorem \ref{Edgeworth_for_Indep_Vars} below. 
For us it yields an expression of $\PP{M_n \leq x \mid \Lambda }$ up
to the accuracy of $o(n^{-1/2})$.

For clarity, we first define relevant notations. 
Let $(X_{ni})_{n\in\NN, i\in\{1, \cdots n\}}$ be a triangular array of random variables with zero means and finite variances, and assume that $X_{n1}, \cdots, X_{nn}$ are independent for all $n\in\NN$.
Furthermore,
\begin{itemize}
\item $\widebar{V}_{n} := n^{-1} \sum_{i=1}^{n} \Var{X_{ni}}$ is positive for all sufficiently large $n$.
\item $\bar{\chi}_{v, n}$ is the average $v^{th}$ cumulant of $\widebar{V}_n^{-1/2} X_{ni}$'s, for $v\in\NN$.
\item $C_n(t) := \EE{\exp(it \widebar{V}_n^{-1/2} \sum_{j=1}^{n} X_{ni})}$.
\item For $v\in\NN$,
\begin{align*}
Q_{vn}(x) &:= \sum_{w=1}^{v} \frac{1}{w!} \p{ \sum_{*(w, v)} \prod_{k=1}^{w} \frac{\bar{\chi}_{j_k + 2, n}}{(j_k + 2)!}}  (-1)^{v + 2w} \frac{d^{v + 2w}}{dx^{v + 2w}} \Phi (x),
\end{align*}
where the summation $*(w, v)$ is over $\{(j_1, \cdots, j_w) \in \NN^w \mid j_1 + \cdots + j_w = v\}$.
\end{itemize}
One verifies that $Q_{vn}(x)$ 
is a product of $\phi(x)$ and a degree-$(3v-1)$ 
polynomial of $x$ with coefficients being polynomials of $\bar{\chi}_{j, n}, j \in \{3, \cdots, v+2\}$.
Further, $Q_{vn}$ is even for odd $v$ and odd for even $v$.
\begin{theorem} \label{Edgeworth_for_Indep_Vars}
For fixed $k \geq 3$, $l \geq 0$ and for $(X_{ni})_{n\in\NN, i\in\{1, \cdots n\}}$, assume that there exist $r_1(k), r_2(n ; k, \tau), r_3(n ; k, l, \epsilon)$ satisfying the following regularity conditions :
\begin{enumerate} [label=\textbf{{R}{\arabic*}}]
	\item \label{r1} For all sufficiently large $n\in\NN$, 
	\[
	n^{-1} \widebar{V}_n^{-k/2} \sum_{i=1}^{n} \EE{|X_{ni}|^k} \leq r_1(k) < \infty.
	\]
	
 	\item \label{r2} For some $\tau \in (0, 1/2)$, 
	\[
	n^{-1} \widebar{V}_n^{-k/2} \sum_{i=1}^{n} \EE{ I (\widebar{V}_n^{-1/2} |X_{ni}| > n^{\tau}) |X_{ni}|^k } 
	\leq r_2(n ; k, \tau) = o(1).
	\]
	\item \label{r3} A generalized Cramer's condition
	\[
	n^{(k + l -2)/2} \int_{|t| > \epsilon } |t|^{l-1} |C_n(t)| dt \leq r_3(n ; k, l, \epsilon) = o(1)
	\]
	holds for some $\epsilon \in (0, 3/(4H_3) )$ and all $n > n_3(k, l, \epsilon)$, where $H_3 := r_1(k)^{3/k} < \infty$ is an upper bound of the average third absolute moments(by power mean inequality).
\end{enumerate}
Then, there exists $N = N(k, l, \tau, \epsilon, n_3)$ such that for $n > N$, the inequality
\begin{align*}
\Bigg|
\frac{d^l}{dx^l} \mathbb{P}( n^{-1/2} \widebar{V}_n^{-1/2} \sum_{i=1}^{n} X_{ni} \leq x ) - \frac{d^l}{dx^l} (\Phi(x) + \sum_{v=1}^{k-2} n^{-v/2} Q_{vn}(x))
\Bigg| 
\leq n^{-(k-2)/2}  \delta(n)
\end{align*}
 holds for all $x\in\RR$.
Here $\delta(n) = o(1)$ depends only on $n, k, l, \tau, \epsilon, r_1(k), r_2(n; k, \tau)$ and $r_3(n; k, l, \epsilon)$.
\end{theorem}

Our reason for presenting this theorem along with the explicit
dependence of the constants is that it provides a uniform bound on
the (derivatives of) difference between the distribution function and
corresponding Edgeworth expansion for all sufficiently large $n$.
Also, we briefly comment on the regularity conditions :
\ref{r1} is about boundedness of $\bar{\chi}_{v, n}, v = 3, \cdots, k$, while \ref{r2}, \ref{r3} are related to tail behavior ; in particular, \ref{r2} resembles the Lindeberg condition for the CLT.

Back to our problem, we state a special case of \Cref{Edgeworth_for_Indep_Vars} when $k=3$ and $l = 0$.

\begin{corollary}
\label{1stEdgeworthSumIndVars}
For $(X_{ni})_{n\in\NN, i\in\{1, \cdots n\}}$ satisfying \ref{r1}, \ref{r2} and \ref{r3} for $k = 3$ and $l = 0$,
\[
\mathbb{P} (n^{-1/2} \widebar{V}_n^{-1/2} \sum_{i=1}^{n} X_{ni} \leq x)
= \Phi(x) + n^{-1/2} \bar{\chi}_{3, n} (1 - x^2)\phi(x) / 6 + o(n^{-1/2}), 
\]
uniformly in $x\in\RR$.
\end{corollary}

Now from \eqref{z_Lambda} and \eqref{S_n_def}, observe that conditioned on $\Lambda$,
$S_n( ( 1  + n^{-1/2} x_n h_n) g_n )$ is a sum of independent random variables. 
That is, \Cref{1stEdgeworthSumIndVars} is applicable for $X_{ni} = c_{ni} (z_i^2 - 1)$ where $c_{ni} := (1+n^{-1/2} x_n h_n(\lambda_i)) g_n(\lambda_i)$, so long as the corresponding regularity conditions \ref{r1}, \ref{r2} and \ref{r3} hold. 
In the moments analysis below, we show that this is the case on $E_{0,n}$ with the \textit{same} $r_1(k), r_2(n ; k, \tau), r_3(n ; k, l, \epsilon)$, and $n_3(k, l, \epsilon)$.

\textbf{Moments analysis.}
Note that $(z_i^2 - 1)$ are mean zero i.i.d. with the characteristic function $\exp(-i\theta)(1-2 i\theta)^{-1/2}$, and so the $k^{th}$ cumulant is $\kappa_k = 2^{k-1} (k-1)!$ for $k\in\NN$.
In particular, adopting the notations above, we have
\begin{align*}
\widebar{V}_n &= 2 n^{-1} \sum_{i=1}^{n} c_{ni}^2, \quad
\bar{\chi}_{k, n} = \kappa_{k}\widebar{V}_n^{-k/2} n^{-1} \sum_{i=1}^{n} c_{ni}^k, \quad
|C_n(t)| = \prod_{i=1}^{n} (1 + 4 \widebar{V}_n^{-1} c_{ni}^2 t^2)^{-1/4}.
\end{align*}
We will show that there exists a positive $C$ such that
\begin{equation} \label{moment_ineq}
C \max_{i=1, \cdots, n} c_{ni}^2 \leq \widebar{V}_n
\end{equation}
for all $x \in \RR$ on $E_{0,n}$, for all sufficiently large $n$.
Note that $c_{ni}$ depends on $x$.
Let us assume \eqref{moment_ineq} for now and verify that \ref{r1},
\ref{r2} and \ref{r3} hold uniformly in $x\in\RR$ on $E_{0,n}$.
First, 
\[
n^{-1} \widebar{V}_n^{-k/2} \sum_{j=1}^{n} \EE{|X_{nj}|^k} 
=  \widebar{V}_n^{-k/2} n^{-1} \sum_{i=1}^{n} |c_{ni}|^k \EE{|z_1^2 - 1|^k}
\leq C^{-k/2} \EE{|z_1^2 - 1|^k},
\]
hence \ref{r1} holds with $r_1(k) = C^{-k/2} \EE{|z_1^2 - 1|^k}$ for all $k\in\NN$. 
Now use the Markov inequalities and then \ref{r1} to get
\[
n^{-1} \widebar{V}_n^{-k/2} \sum_{i=1}^{n} \EE{ I(\widebar{V}_n^{-1/2} |X_{ni}| > n^{\tau}) |X_{ni}|^k }
\leq n^{-\tau-1}  \widebar{V}_n^{-(k+1)/2} \sum_{i=1}^{n} \EE{|X_{ni}|^{k+1}} \leq n^{-\tau} r_1(k+1),
\] 
which shows that \ref{r2} holds with $r_2(n ; k, \tau) = n^{-\tau} r_1(k+1)$ for any $\tau \in (0, 1/2)$ and $k\in\NN$. 

For any $m\in \{1, \cdots, n\}$, define 
$s_m := \sum_{1\leq i_1 < \cdots < i_m \leq n} \prod_{j=1}^{m} c_{ni_j}^2 $ and $n_m := n^m - n!/(n-m)!$.
We then have 
\begin{align*}
(n \widebar{V}_n/2)^m & = (\sum_{i=1}^n c_{ni}^2)^m = \sum_{1 \leq
                        i_1, \cdots, i_m \leq n} \prod_{j=1}^{m}
                        c_{ni_j}^2 \\ 
& \leq n_m \max_{i=1, \cdots, n} c_{ni}^{2m} + m! s_m \leq C^{-m} n_m \widebar{V}_n^m +  m! s_m,
\end{align*}
so that $ ( 2\widebar{V}_n^{-1})^m s_m \geq (n^m - (2 C^{-1}) ^m n_m) / m!$.
Hence 
\begin{align*}
\prod_{i=1}^{n} (1 + 4 \widebar{V}_n^{-1} c_{ni}^2 t^2)
&\geq (4 \widebar{V}_n^{-1}t^2)^m s_m
\geq (2n t^2)^m (1 - (2 C^{-1}) ^m n_m/n^m) / m!.
\end{align*}
Now $\lim_{n\to\infty} n_m / n^m = 0$ for any fixed $m\in\NN$, so,  with $m = 4(k+l)$, it follows that $ |C_n(t)| \leq 2 (m!)^{1/4} (2n t^2)^{-(k+l)} $ for all $n  > n_3(k, l, \epsilon)$.
This implies \ref{r3} with 
$r_3(n ; k, l, \epsilon) = 2^{-(k+l-2)}(4(k+l)!)^{1/4} n^{-(k+l+2)/2} \epsilon^{- (2k + l)} / (2k+l)$ 
for any $\epsilon \in (0, 3/(4H_3))$ and $ k \geq 3, l \geq 0$.

\begin{proof} [Proof of \eqref{moment_ineq}]
Throughout the proof, $n > n_0(\delta)$ and $\Lambda \in E_{0,n}$ are assumed, so that $\lambda_i \in [0, \rho)$,
$ g_n(\lambda_i) = (\rho_n  - \lambda_i)^{-1} \in [\rho_n^{-1}, \delta^{-1}] $ and 
$|h_n(\lambda_i)| = |r_n - \lambda_i g_n(\lambda_i)| \leq \max(r_n, \rho \delta^{-1})$.
Consequently,
\[
|c_{ni}| = | 1+n^{-1/2} x_n h_n(\lambda_i) | g_n(\lambda_i) \leq \delta^{-1} ( 1 + \max(r_n, \rho \delta^{-1}) |n^{-1/2}  x_n| ), 
\]
so that $\max_{i=1, \cdots, n} c_{ni}^2 \leq C_1 ( 1 + C_2 |n^{-1/2}  x_n| )^2$ for positive constants $C_1, C_2$ independent of $n$ and $x$.
Therefore, it suffices to show that there exists a positive $\epsilon$ such that
\begin{align} \label{moment_ineq_pf1}
\epsilon (1 + C_2 |n^{-1/2}  x_n|)^2 \leq \widebar{V}_n /2,
\end{align}
for all $ x_n \in \RR$. 
Let $v_k = \textsf{F}_n( g_n^2 h_n^k )$ for $k = 0, 1, 2$, and then write $ \widebar{V}_n / 2 = v_2 (n^{-1/2}  x_n)^2 + 2 v_1 (n^{-1/2}  x_n) + v_0$.
Hence \eqref{moment_ineq_pf1} is equivalent to 
\[
2 (C_2\epsilon - v_1 \text{sign}(x_n) )|n^{-1/2}  x_n| \leq (v_2 - \epsilon C_2^2) (n^{-1/2}  x_n)^2 + (v_0 - \epsilon) 
\]
for all $x_n \in \RR$.
In view of the AM-GM inequality and its equality condition, this is equivalent to $0\leq (v_0 -\epsilon), (v_2 - \epsilon C_2^2)$ and $(v_1 + C_2 \epsilon)^2 \leq (v_2 - C_2^2 \epsilon) (v_0 - \epsilon)$.
But then the first and the third inequalities yield the second, so the desired condition is
 $$\epsilon \in (0, \min ( v_0, (v_2 v_0 - v_1^2) (v_0 C_2^2 + 2 v_1 C_2 + v_2)^{-1} )) ).$$
This is true when 
\begin{equation} \label{V_discriminant_ineq}
v_2 v_0 - v_1^2 \geq C_4
\end{equation}
for a positive $C_4$,
because 
$v_0 \geq 1$, $v_0 C_2^2 + 2 v_1 C_2 + v_2 = v_0 (C_2 + v_1/v_0)^2 + (v_0 v_2 - v_1^2) / v_0$ is positive when \eqref{V_discriminant_ineq} holds, and bounded above on $E_{0,n}$.
Finally, since $( \sum a_i^2 ) ( \sum b_i^2 ) - (\sum a_i b_i)^2 = \sum_{i < j } (a_i b_j - a_j b_i)^2$ 
and $ h_n(\lambda') - h_n(\lambda) = \lambda g_n(\lambda) - \lambda' g_n(\lambda') = \rho_n g_n(\lambda) g_n(\lambda') (\lambda - \lambda')$, we have
\begin{align*}
v_2 v_0 - v_1^2 & = \ssf_n(g_n^2 h_n^2) \ssf_n(g_n^2) - \ssf_n(g_n^2 h_n)^2
= n^{-2} \sum_{1\leq i < j \leq n} ( g_n(\lambda_i) g_n(\lambda_j) )^2 
\p{ h_n(\lambda_i)- h_n(\lambda_j) }^2 \\
&= \rho_n^2 n^{-2} \sum_{1\leq i < j \leq n} \p{ g_n(\lambda_i) g_n(\lambda_j) }^4
\p{ \lambda_i - \lambda_j }^2 \\
&\geq \rho_n^{-6} n^{-2} \sum_{1\leq i < j \leq n} \p{ \lambda_i - \lambda_j }^2 
= \rho_n^{-6} (\ssf_n(m_2) - \ssf_n(m_1)^2) \geq  (\rho + \gamma)^{-6} \gamma^2 / 8,
\end{align*}
so we have shown \eqref{moment_ineq}, and consequently the claim.
\end{proof}

\textbf{First order Edgeworth expansion for} $M_n$.
From Corollary \ref{1stEdgeworthSumIndVars} and \eqref{M_n}, we have
\begin{align*}
\EE[n]{ \PP{M_n \leq 2 \sigma_n^{-1} x  \mid \Lambda} - ( \Phi(y_n) + n^{-1/2} \widebar{V}_n^{-3/2} \bar{\kappa}_{3, n} ( 1 - y_n^2) \phi(y_n) / 6 )} =  o(n^{-1/2})
\end{align*}
uniformly in $x\in\RR$, where
$ y_n := \widebar{V}_n^{-1/2} (2 \sigma_n^{-1} x - n^{-1/2} G_n(g_n)) $ and $\bar{\kappa}_{3,n} = 8n^{-1}\sum_{i=1}^n c_{ni}^3$.
It then suffices to show that
\begin{align} \label{M_n_Edgeworth}
& \EE[n]{ \Phi(y_n) + n^{-1/2} \widebar{V}_n^{-3/2} \bar{\kappa}_{3, n} ( 1 - y_n^2) \phi(y_n) / 6 } \nonumber \\
= \quad & \Phi(x) + n^{-1/2} (\kappa_{2,n}^{-3/2} \kappa_{3,n}  (1 - x^2) / 6 - \kappa_{2,n}^{-1/2}  \mu(g_n) ) \phi(x) + o(n^{-1/2}),
\end{align}
uniformly in $x \in \RR$.
To this end, we introduce the following notation.
\begin{definition}
For $\alpha > 0$ and a polynomial $p_n(t)=\sum_{i=0}^{k} c_{ni} t^i$ with random coefficients $c_{ni}$'s, 
$p_n$ is $PO(n^{-\alpha} ; E_{0,n})$ if
$ \EE[n]{ |c_{ni}| } = O(n^{-\alpha}), i = 0, \cdots, k. $
\end{definition}
With this definition, we will show that 
\begin{equation} \label{PO}
\widebar{V}_n  - \kappa_{2, n} = PO(n^{-1} ; E_{0,n}), 
\quad \bar{\kappa}_{3,n} - \kappa_{3,n} = PO(n^{-1/2} ; E_{0,n}),
\end{equation}
when both are treated as polynomials of $x_n = \rho_n^{-1} \sigma_n x$.
To prove the first part, observe that
\begin{align*}
\widebar{V}_n  - \kappa_{2,n }
&= 2 (v_2 (n^{-1/2}  x_n)^2 + 2 v_1 (n^{-1/2}  x_n) + v_0- \textsf{F}_{\gamma_n}( g_n^2)  ) 
\\ & =  2 n^{-1} ( v_2 x_n^2 + 2 n^{-1/2} G_n( g_n^2 h_n) x_n + G_n( g_n^2) ),
\end{align*}
where the second equality uses
$
v_1 - n^{-1} G_n(g_n^2 h_n) = \ssf_{\gamma_n}(g_n^2 h_n) = r_n \ssf_{\gamma_n}(g_n^2) -  \ssf_{\gamma_n}(m_1 g_n^3) = 0,
$
from \eqref{r_n}.
Also, it is clear that $g_n^2 h_n$ and $g_n^2$ satisfy the sufficient
conditions for \Cref{Cor_TailBound_LSS}, hence
\Cref{Cor_limiting_moment} implies that $(\widebar{V}_n  -
\kappa_{2,n})$ is $PO(n^{-1} ; E_{0,n})$. 
The second part of \eqref{PO} can be proved in a similar yet simpler way ; namely,
\begin{align*}
	 \bar{\kappa}_{3,n} - \kappa_{3,n}
	&= 8 n^{-1/2} (n^{-1} u_3 x_n^3 + 3 n^{-1/2} u_2 x_n^2 + 3 u_1 x_n +  n^{-1/2} G_n(g_n^3) ),
\end{align*}
where $u_k = \ssf_n(g_n^3 h_n^k) , k = 1, 2, 3$. These are also absolutely bounded on $E_{0,n}$.

To exploit \eqref{PO}, we introduce a trivial inequality and its consequence as follows.
\begin{proposition} \label{poly_exp_ineq}
	For any univariate polynomial $p$(with deterministic coefficients) and a positive $s$, there exists a constant $C(p, s)$ such that $| p(t) \exp( -s t^2) | \leq C(p, s)$ for all $t \in \RR$.
\end{proposition}
\begin{corollary} \label{poly_exp_ineq_cor}
	If $p_n$ is $PO(n^{-\alpha} ; E_{0,n})$ for some $\alpha > 0$, then for any positive $s$, 
	$$\sup_{t\in\RR}| \EE[n]{ p_n(t) \exp( -s t^2) } | = O(n^{-\alpha}).$$
\end{corollary}

Now we show
\begin{align}
\EE[n]{\Phi(y_n) - \Phi(x) + n^{-1/2} \kappa_{2,n}^{-1/2} \mu(g_n) \phi(x) }   &= o(n^{-1/2}), \label{M_n_Edgeworth_proof_1} \\
\EE[n]{ \widebar{V}_n^{-3/2} \bar{\kappa}_{3,n} ( 1 - y_n^2 ) \phi(y_n) -  \kappa_{2,n}^{-3/2} \kappa_{3,n} ( 1-x^2 ) \phi(x)  }
& = O(n^{-1/2}) \label{M_n_Edgeworth_proof_2}
\end{align}
uniformly in $x \in \RR$, which implies \eqref{M_n_Edgeworth} along with \Cref{poly_exp_ineq} and the tail bound on $E_{0, n}$.
These are fairly easy to prove on any compact subset of $\RR$, but for uniform convergence, the proof is more delicate, due to the dependence of $\widebar{V}_n$ and $\bar{\kappa}_{3,n}$ on $x$.
Although a wide interval of $x$ would be practically meaningful, we prove uniform convergence here.

\begin{proof}[Proof of \eqref{M_n_Edgeworth_proof_1} and \eqref{M_n_Edgeworth_proof_2}]
Observe that on $E_{0,n}$, $\widebar{V}_n$ and $\kappa_{2, n}$ are bounded below by a positive constant uniformly in $x\in\RR$, in view of $\widebar{V}_n \geq v_2^{-1} (v_0 v_2 - v_1^2)$ and \eqref{V_discriminant_ineq}.
On the other hand, by the AM-GM inequality and \eqref{V_discriminant_ineq}, we have the upper bound
\begin{equation} \label{V_n_upper}
\widebar{V}_n  \leq 4(n^{-1} v_2 x_n^2 + v_0).
\end{equation}
Now we can prove \eqref{M_n_Edgeworth_proof_1} as follows : let $\alpha_n = \widebar{V}_n^{-1/2} \kappa_{2,n}^{1/2}$, then it suffices to show that
\begin{align}
&\EE[n]{ \Phi(y_n) - \Phi(\alpha_n x) + n^{-1/2} \widebar{V}_n^{-1/2} G_n(g_n) \phi(\alpha_n x) }, \label{M_n_Edgeworth_proof_1_1} \\
&\EE[n]{ \Phi(\alpha_n x) - \Phi(x) } , \label{M_n_Edgeworth_proof_1_2} \\
&\EE[n]{ n^{-1/2} \widebar{V}_n^{-1/2} G_n(g_n) ( \phi(\alpha_n x) - \phi(x) ) }, \quad \text{and} \label{M_n_Edgeworth_proof_1_3} \\
&\EE[n]{ n^{-1/2} \kappa_{2,n}^{-1/2} \p{ \alpha_n - 1 } G_n(g_n) \phi(x) }\label{M_n_Edgeworth_proof_1_4}
\end{align}
are $O(n^{-1})$ uniformly in $x\in\RR$,
because $\EE[n]{ G_n(g_n) - \mu(g_n) } = o(1)$ from \Cref{Cor_limiting_moment}.
From the second order Taylor expansion of $\Phi(y_n)$ centered at $\alpha_n x$ and using \Cref{poly_exp_ineq}, \eqref{M_n_Edgeworth_proof_1_1} is $O( n^{-1} \EE[n]{ G_n(g_n)^2 } )$, and hence $O(n^{-1})$ uniformly in $x\in\RR$, by \Cref{Cor_limiting_moment}.
Next, for \eqref{M_n_Edgeworth_proof_1_2} and \eqref{M_n_Edgeworth_proof_1_3}, we consider two cases :

(case 1) $ x^2 \leq n $ : 
This assumption implies that $\widebar{V}_n$ is bounded above by a positive constant on $E_{0,n}$, by \eqref{V_n_upper}.
Therefore, on $E_{0,n}$, $\alpha_n$ is bounded below by a positive $\alpha_0$, and thus $\exp(- s t^2) \leq \exp ( - s \beta_0^2 x^2 ) $ for all $t$ between $x$ and $\alpha_n x$ and for all positive $s$, where $\beta_0 = \min(\alpha_0, 1)$.
Using this fact, $| t | \exp (-t^2/2) \leq \exp (-t^2/ 4)  $, and the first order Taylor expansions of $\Phi(\alpha_n x)$ and $\phi(\alpha_n x)$ centered at $x$, it follows that \eqref{M_n_Edgeworth_proof_1_2}, \eqref{M_n_Edgeworth_proof_1_3} are 
\begin{align*}
O(\EE[n]{ | (\alpha_n - 1 ) x | \exp(-\beta_0^2 x^2 /2) } ), \quad 
O( n^{-1/2} \EE[n]{ | G_n(g_n) (\alpha_n - 1 ) x | \exp(- \beta_0^2 x^2 / 4) }),
\end{align*}
respectively.
These are $O(n^{-1})$ uniformly in $x\in[-\sqrt{n}, \sqrt{n}]$, because of
\begin{align}\label{alpha_n_minus_1}
\alpha_n - 1 = \widebar{V}_n^{-1/2}  (\kappa_{2,n} - \widebar{V}_n ) (\widebar{V}_n^{1/2} + \kappa_{2,n}^{1/2})^{-1} 
= PO(n^{-1} ; E_{0,n}),
\end{align}
\Cref{poly_exp_ineq_cor} and the Cauchy-Schwarz inequality(for the second case).

(case 2) $ x^2 > n $ :
In this case we have $\widebar{V}_n = O(n^{-1} x^2) $ on $E_{0, n}$ from \eqref{V_n_upper}.
Then $|\alpha_n x|^{-1} = O(n^{-1/2})$ on $E_{0, n}$ uniformly in $x\in[-\sqrt{n}, \sqrt{n}]^c$, and hence from $0 < 1 - \Phi(|t| ) \leq \phi(|t|) / |t| = O( |t|^{-2})$,
we conclude that $1-\Phi(| x |)$, $1-\Phi(| \alpha_n x |)$, $\phi(x)$, $\phi(\alpha_n x)$  are all $O(n^{-1})$ uniformly in $x\in[-\sqrt{n}, \sqrt{n}]^c$, and so the same is true for \eqref{M_n_Edgeworth_proof_1_2}, \eqref{M_n_Edgeworth_proof_1_3}.

Combining these cases gives the desired result for \eqref{M_n_Edgeworth_proof_1_2} and \eqref{M_n_Edgeworth_proof_1_3}.
Furthermore, \eqref{M_n_Edgeworth_proof_1_4} immediately follows from \eqref{alpha_n_minus_1}, \Cref{poly_exp_ineq_cor} and the Cauchy-Schwarz inequality.

In a similar manner to the proof of \eqref{M_n_Edgeworth_proof_1} just given, we can decompose the RHS of \eqref{M_n_Edgeworth_proof_2} into
\begin{align}
&\EE[n]{ \widebar{V}_n^{-3/2} \bar{\kappa}_{3,n} ( ( 1 - z_n^2 ) \phi(z_n) -  ( 1- (\alpha_n x)^2 ) \phi(\alpha_n x)) }, 
 \label{M_n_Edgeworth_proof_2_1} \\
 &\EE[n]{ \widebar{V}_n^{-3/2} \bar{\kappa}_{3,n} ( ( 1- (\alpha_n x)^2 ) \phi(\alpha_n x)  -  ( 1-x^2 ) \phi(x) ) }, 
 \label{M_n_Edgeworth_proof_2_2} \\
 &\EE[n]{ \widebar{V}_n^{-3/2} (\bar{\kappa}_{3,n} - \kappa_{3,n} ) ( 1-x^2 ) \phi(x) }, 
 \label{M_n_Edgeworth_proof_2_3} \\
  &\EE[n]{  \kappa_{2,n}^{-3/2}  ( \alpha_n^3 - 1  ) \kappa_{3,n} ( 1-x^2 ) \phi(x) }, 
 \label{M_n_Edgeworth_proof_2_4} 
\end{align}
which are to be shown to be $O(n^{-1/2})$ uniformly in $x\in\RR$.
From \eqref{moment_ineq_pf1}, $ \widebar{V}_n^{-3/2} | \bar{\kappa}_{3,n} | $ is bounded above uniformly in on $E_{0,n}$,
which leads to the desired result for \eqref{M_n_Edgeworth_proof_2_1} and \eqref{M_n_Edgeworth_proof_2_2} by the same methods as for  \eqref{M_n_Edgeworth_proof_1_1} and \eqref{M_n_Edgeworth_proof_1_3}, with small changes in details ; the first order Taylor expansion suffices for \eqref{M_n_Edgeworth_proof_2_1},
and case 2 for \eqref{M_n_Edgeworth_proof_2_2} requires $0 < (t^2 - 1) \phi(t) \leq 8 t^{-2} $ if $t^2 > 1$.
Finally, \eqref{PO} and \eqref{alpha_n_minus_1} give the desired properties for \eqref{M_n_Edgeworth_proof_2_3} and \eqref{M_n_Edgeworth_proof_2_4}, respectively. 
\end{proof}

\noindent {\bf 3.5. Delta method for Edgeworth expansion} \label{delta_section}

In this section, we prove that $\delta_n x_n$ is ignorable in the sense of \ref{step4}.
The decomposition given in \eqref{linearization} is inspired by the discussion in \citet[Chap.~2.7]{hall1992bootstrap}. 
The delta method is briefly introduced there as follows : for two statistics $U_n$ and $U'_n$ whose limiting distributions are $N(0, 1)$, if $\Delta_n := U_n - U'_n$ is of order $O_p(n^{-j/2})$ for $j \in \NN$, then ``generally'', $\PP{U_n \leq x} - \PP{U'_n \leq x}$ is of order $O(n^{-j/2})$.
Therefore, if the $(j-1)^{th}$ order Edgeworth expansion for $U_n$ is easy to calculate, so is for $U'_n$.
However, neither sufficient conditions nor a rigorous proof for this method is given there.
Furthermore, $\Delta_n$ is linear in $x$ in our case.
Hence, we prove a version of the delta method for Edgeworth expansion in our context.

\begin{proposition} \label{delta_method}
	Suppose that $U_n$ admits the first order Edgeworth expansion
	\[
	\mathbb{P}_n \p{U_n \leq x} = \Phi(x) + n^{-1/2} p_1(x) \phi(x) + o(n^{-1/2})
	\]
	uniformly in $x\in\RR$, for a polynomial $p_1$.
	Also, assume that random variables $J_n$ do not depend on $x$, and satisfy
	$\mathbb{P}_n (|J_n| > n^{-1/2} \epsilon_n) = o(n^{-1/2})$
	for a non-random sequence $\{\epsilon_n\}$ converging to 0.
	Then $$ \mathbb{P}_n \p{ U_n + x J_n \leq x} - \mathbb{P}_n \p{ U_n \leq x} = o(n^{-1/2}) $$
	uniformly in $x\in\RR$.
\end{proposition}

\begin{proof}
	Note that 
	\[
	| \PP[n]{ U_n + x J_n \leq x} - \mathbb{P}_n \p{ U_n \leq x } |
	\leq \mathbb{P}_n (|J_n| > n^{-1/2} \epsilon_n) + \mathbb{P}_n ( |U_n - x| \leq |x| n^{-1/2} \epsilon_n ),
	\]
	hence from the assumption $\mathbb{P}_n(|J_n| > n^{-1/2} \epsilon_n)= o(n^{-1/2})$ it suffices to show that 
	\[
	\mathbb{P}_n ( |U_n - x| \leq |x| n^{-1/2} \epsilon_n ) = o(n^{-1/2})
	\]
	uniformly in $x\in\RR$.
	This follows from the uniform convergence assumption on the first order Edgeworth expansion of $U_n$, and the following inequalities : for $y \in [-1/2, 1/2]$, by \Cref{poly_exp_ineq},
	\begin{align*}
	|\Phi(x(1+y)) - \Phi(x))| &\leq |xy| \max_{z\in[-1/2, 1/2]} \phi(x(1+z)) \leq |xy| \phi(x / 2) 
	= O( |y| ),\\
	|p_1(x(1+y)) \phi(x(1+y)) - p_1(x) \phi(x) | &\leq |xy| \max_{z\in[-1/2, 1/2]} |p_2(x(1+z))| \phi(x(1+z)) \\
	&\leq |xy| |p_2|(|3x/2|)\phi(x/2) = O( |y| ).
	\end{align*}
Here $p_2$ is the polynomial satisfying
$\frac{d}{dx}(p_1(x) \phi(x)) = p_2(x) \phi(x)$, and $|p_2|$ is
the polynomial with coefficients being the absolute values of
coefficients of $p_2$.
\end{proof}
Finally, we prove \eqref{apply_delta_method} using this proposition with $U_n = \sigma_n M_n / 2$, $J_n = \rho_n^{-1} \sigma_n^2 \delta_n / 2$ and $\epsilon_n \asymp n^{-\zeta}$ for any $\zeta \in (0, 1/2)$.
Recall the definition of $\delta_n$ from \eqref{delta_n_nu_n} : $\delta_n = n^{-1} G_n(\lambda g_n^2) - (\nu_n - n^{-1/2} r_n S_n( g_n))$.
As $\mathbb{P}_n( |n^{-1} G_n(m_1 g_n^2) | > n^{-1/2-\zeta} ) = o(n^{-1/2})$ by \Cref{TailBound_LSS}, we only need to consider $(\nu_n - n^{-1/2} r_n S_n( g_n))$.
Observe that from \eqref{key_eq_2} and \eqref{key_eq_denom},
\[
(\lh - \rho_n) (1 + \ell n^{-1} z' \Lambda R^2(\rho_n) z)
= n^{-1/2} \ell \rho_n [S_n(g_n) + n^{-1/2} G_n(g_n)] + (\lh - \rho_n) \nu_n.
\]
Multiply both sides by $ - n^{-1} z' \Lambda R(\lh) R^2(\rho_n) z$ to yield
\begin{equation*}
(1 + \ell n^{-1} z' \Lambda R^2(\rho_n) z) \nu_n = 
- n^{-1/2} \ell \rho_n [S_n(g_n) + n^{-1/2} G_n(g_n)] \cdot n^{-1} z' \Lambda R(\lh) R^2(\rho_n) z + \nu_n^2,
\end{equation*}
because of \eqref{nu_n}. Consequently, on $E_{0,n}$ we have
\begin{align*}
&|\nu_n - n^{-1/2} r_n S_n( g_n)|
\leq (1 + \ell n^{-1} z' \Lambda R^2(\rho_n) z)  | \nu_n - n^{-1/2} r_n S_n( g_n) |
\\ \leq & \ n^{-1/2} \ell \rho_n | S_n(g_n) | \cdot | n^{-1} z' \Lambda R(\lh) R^2(\rho_n) z + (1 + \ell n^{-1} z' \Lambda R^2(\rho_n) z) \tfrac{r_n}{\ell \rho_n} |
\\ &
+ n^{-1} \ell \rho_n | G_n(g_n) | \cdot  | n^{-1} z' \Lambda R(\lh) R^2(\rho_n) z | + \nu_n^2.
\end{align*}
Furthermore, the following holds from \eqref{r_n}, the resolvent identity and \eqref{stoc_decomp}
\begin{align*}
& n^{-1} z' \Lambda R(\lh) R^2(\rho_n) z
+ (1 + \ell n^{-1} z' \Lambda R^2(\rho_n) z) \tfrac{r_n}{\ell \rho_n} 
\\ = & n^{-1} z' \Lambda R(\lh) R^2(\rho_n) z 
+ (1 + \ell n^{-1} z' \Lambda R^2(\rho_n) z) \ssf_{\gamma_n}(m_1 g_n^3) / (1+ \ell \ssf_{\gamma_n}(m_1 g_n^2))
\\= &
(\lh - \rho_n) n^{-1} z' \Lambda R(\lh) R^3(\rho_n) z  + n^{-1} z' \Lambda R^3(\rho_n) z
+ \ssf_{\gamma_n}(m_1 g_n^3) +  (n^{-1} z' \Lambda R^2(\rho_n) z - \ssf_{\gamma_n}(m_1 g_n^2))\tfrac{r_n}{\rho_n}
\\=&
(\lh - \rho_n) n^{-1} z' \Lambda R(\lh) R^3(\rho_n) z
+ n^{-1/2} S_n(m_1 g_n^2(\tfrac{r_n}{\rho_n} - g_n)) + n^{-1} G_n(m_1 g_n^2(\tfrac{r_n}{\rho_n}- g_n)).
\end{align*}
Now considering that $\ell, \rho_n, r_n, \|\Lambda\|_{\infty},$ $\|R(\lh)\|_{\infty}, \|R(\rho_n)\|_{\infty}$ are absolutely bounded on $E_{0,n}$ for $n > n_0(\delta)$, and $\nu_n = -(\lh - \rho_n) n^{-1} z' \Lambda R(\lh) R^2(\rho_n) z$ \eqref{nu_n}, it suffices to show that
\begin{align*}
\mathbb{P}_n( |S_n(g_n) | > n^{1/4 - \zeta/2} ), \quad 
\mathbb{P}_n(| \lh - \rho_n | > n^{-1/4 - \zeta/2}),  \quad
\mathbb{P}_n( n^{-1} z' z > 2), \quad
\mathbb{P}_n( |G_n( g_n ) | > n^{1/2 - \zeta} )
\end{align*}
are of probability $o(n^{-1/2})$ for any $\zeta \in (0, 1/2)$. 
Each such bound can be easily deduced from \Cref{TailBound_Indep}, \Cref{TailBound_LSS} and \Cref{Cor_l_hat}.


\setcounter{section}{4} 
\setcounter{subsection}{1} 
\setcounter{equation}{0} 

\noindent {\bf 4. Discussion}

This study clearly leaves some natural questions for further
research. 
We considered a single supercritical spike; extension to a finite
number of separated simple supercritical eigenvalues is presumably
straightforward.
Less immediately clear is the situation with a supercritical
eigenvalue of multiplicity $K > 1$, as the limiting distribution for
the associated $K$ eigenvalues is $GOE(K)$ rather than ordinary
Gaussian.

A common use of Edgeworth approximations is to improve the coverage
properties of confidence intervals based on Gaussian limit theory. In
ongoing work, we are exploring such improvements for one- and
two-sided intervals for $\ell$. 

Development of a second order Edgeworth approximation (kurtosis
correction) would appear to require a first order or skewness
correction for certain linear statistics in the Bai-Silverstein
central limit theorem, which is not yet available. 

We assumed that the observations $x_j$ were Gaussian and that
assumption is used in an important way to create the i.i.d. variates 
$z = (z_i) = U'Z_1$, independent of the noise eigenvalues $\Lambda$,
as input to the conditional Edgeworth expansion. Thus extension of the
results to non Gaussian $x_j$ is an open issue for future work.

\vskip 14pt
\noindent {\large\bf Supplementary Materials}

\newcounter{suppsection}
\setcounter{suppsection}{0}

\newenvironment{suppequation}{%
	\renewcommand\theequation{S\thesuppsection.\arabic{equation}}
	\begin{equation}}
{\end{equation}}


We provide proofs of identities and propositions used in the main text.



\noindent{\bf S1. Identities}
\refstepcounter{suppsection}

\noindent{\bf S1.1. Expectations with respect to Marchenko-Pastur distribution}

Let $\gamma \in (0,\infty), \ell > 1+\sqrt{\gamma}$ and $\rho_n 
= \ell + \ell \gamma_n / (\ell - 1)$. Then
\begin{suppequation} \label{rho_n_deriv}
\frac{\partial \rho_n}{\partial \ell} = \frac{(\ell-1)^2 - \gamma_n}{(\ell-1)^2}.
\end{suppequation}
Also, the Stieltjes transform of the companion Marchenko-Pastur distribution is given by
\[
\ssf_{\gamma}(f_z)
 = (- z + \gamma - 1 + \sqrt{(z-\gamma - 1)^2 - 4 \gamma})/ (2z), 
\quad \forall z \in (b(\gamma), +\infty)
\]
where $f_z(\lambda) := ( \lambda - z)^{-1}$, from equantion (2.8) of \cite{yao2015sample}. 
Substituting $\gamma_n$ into $\gamma$ and $\rho_n$ into $z$(which is possible since $\rho_n > (1+\sqrt{\gamma_n})^2$), it follows that 
\begin{suppequation} \label{F_g_n}
\ssf_{\gamma_n}(g_n) = \ell^{-1}.
\end{suppequation}
Taking partial derivatives of \eqref{F_g_n} with respect to $\ell$, along with \eqref{rho_n_deriv}, gives
\begin{suppequation} \label{F_g_n_2}
\ssf_{\gamma_n}(g_n^2) = (1 - \ell^{-1})^2 ((\ell-1)^2 - \gamma_n)^{-1} =  2 \sigma_n^{-2}
\end{suppequation}
and 
\begin{suppequation} \label{F_g_n_3}
\ssf_{\gamma_n}(g_n^3) =(1-\ell^{-1})^3 ((\ell-1)^3 + \gamma_n) ((\ell-1)^2 - \gamma_n)^{-3},
\end{suppequation}
as desired. 

\noindent{\bf S1.2. Explicit expressions of $\mu(g)$ and $\mu(g_n)$}

We use the formula (5.13) in \cite{bai2004clt}.
First, by $x = 1+\gamma+2\sqrt{\gamma}\cos\theta$,
\begin{align*}
\int_{a(\gamma)}^{b(\gamma)} \frac{g(x)}{\sqrt{4\gamma - (x - 1 -\gamma)^2}} dx
& = \int_{-\pi}^{0} \frac{g(1 + \gamma + 2\sqrt{\gamma} \cos \theta)}{\sqrt{1-\cos^2 \theta}} (-\sin \theta)d\theta \\
& = \frac{1}{2} \int_{-\pi}^{\pi} g(1 + \gamma + 2\sqrt{\gamma} \cos \theta) d\theta.
\end{align*}
Then, letting $z = \exp(i\theta)$ gives
\begin{align*}
\int_{-\pi}^{\pi} g(1 + \gamma + 2\sqrt{\gamma} \cos \theta) d\theta
&= \oint_{|z| = 1} g(1 + \gamma + \sqrt{\gamma} (z+z^{-1})) (iz)^{-1} dz \\
&= i \oint_{|z| = 1} (\sqrt{\gamma} z^2 - (\ell-1 + \gamma(\ell-1)^{-1}) z + \sqrt{\gamma})^{-1} dz \\
&= i \oint_{|z| = 1} (z - \sqrt{\gamma}(\ell-1)^{-1})^{-1} \p{\sqrt{\gamma} z - (\ell-1)}^{-1} dz \\
&= - 2\pi (\gamma(\ell-1)^{-1} - (\ell-1))^{-1} 
\\ & = 2\pi (\ell-1) \p{\ell-1-\sqrt{\gamma}}^{-1} \p{\ell-1 + \sqrt{\gamma}}^{-1}
\end{align*}
by Cauchy integral formula with the assumption $\ell-1 > \sqrt{\gamma}$.
Meanwhile,
$
g((1\pm\sqrt{\gamma})^2) = (\rho(\ell, \gamma) - (1\pm\sqrt{\gamma})^2)^{-1}
= (\ell-1) \p{\ell-1 \mp \sqrt{\gamma}}^{-2},
$
hence
\begin{align*}
\mu(g) &= (\ell-1) ((\ell-1-\sqrt{\gamma})^{-1} - (\ell-1+\sqrt{\gamma})^{-1})^2 / 4
= \gamma (\ell-1) ((\ell-1)^2 - \gamma)^{-2},
\end{align*}
as desired.
The corresponding expression for $\mu(g_n)$
\begin{suppequation} \label{mu_g_n}
\mu(g_n) = \gamma_n (\ell-1) ((\ell-1)^2 - \gamma_n)^{-2},
\end{suppequation}
is available when $\ell - 1 > \sqrt{\gamma_n}$ i.e. for large enough $n$.

\textbf{Remark.} Although the formula (5.13) is derived only for $\gamma \leq 1$ in \cite{bai2004clt}, the following identity
\[
G_n(f) = \sum_{i=1}^{p} f( \lambda_i) - p F_{\gamma_n}(f)
= \sum_{i=1}^{n} \tilde{f}_n( \tilde{\lambda}_i) - n F_{\gamma_n^{-1}}(\tilde{f}_n)
=: \textsf{G}_p(\tilde{f}_n),
\]
where $\tilde{f}_n (\lambda) := f(\gamma_n \lambda)$ and $\tilde{\lambda}_i := \gamma_n^{-1} \lambda_i $,
turns the setting
\[ n, \quad p, \quad \gamma_n, \quad n^{-1} Z'_2 Z_2, \quad f\]
into
\[ p, \quad n, \quad \gamma_n^{-1}, \quad p^{-1} Z_2 Z'_2, \quad \tilde{f}_n.\]
Thus, along with \Cref{G_n_lemma}(which is proved below), this correspondence gives the same formula for $\gamma > 1$.

\noindent{\bf S2. Propositions}
\refstepcounter{suppsection}

	\noindent{\bf S2.1. \Cref{TailBound_Indep}} 
	
	We can prove and use results in Example 2.4 of \cite{wainwright2015basic} : the moment generating function of $(z_0^2 - 1)$ where $z_0\sim N(0,1)$ is given by
	\[
	\EE{\exp (\theta(z_0^2-1)) } 
	= \exp (-\theta) ( 1-2\theta )^{-1/2} = \exp(\sum_{k=2}^{\infty} 2^{k-1}\theta^k / k)
	\]
	for $\theta < 1/2$, and is bounded by $ \exp (2 \theta^2) $ for $\theta \in [-1/4, 1/4]$, because
	\begin{align*}
	2 \theta^2 - \sum_{k=2}^{\infty} 2^{k-1}\theta^k / k
	&= \theta^2 (1 - \sum_{k=3}^{\infty} 2^{k-1}\theta^{k-2} / k)
	 \geq \theta^2 (1 - \sum_{k=3}^{\infty} 2^{-k+3} / k)
	= \theta^2 (6 - 8 \log 2) > 0.
	\end{align*}
	Combining this, Markov inequality and independence of $z$ and $\Lambda$, it follows that
	\begin{align*}
	 \mathbb{P} (E_{1,n}^{c} \cap E_{2,n}(f, M) \mid \Lambda)
	 &\leq  \EE{I(E_{1,n}^c) \p{\exp (S_n(f / U_f))+ \exp (-S_n(f / U_f)) } 
		\mid \Lambda}  \exp (-M/ U_f)
	\\ & \leq 2 \exp(2 \ssf_n( f^2 / U_f^2) ) \exp (-M/ U_f) \leq 15 \exp (-M/ U_f), 
	\end{align*}
	which directly implies
	$ \mathbb{P}(E_{1,n}^{c} \cap E_{2,n}(f, M)) \leq 15\exp (-M/ U_f) $,
	as desired.
	
	\noindent{\bf S2.2. \Cref{TailBound_LSS}} 

	Let $ \textsf{f}_n(\lambda) := f_n \p{ (\lambda \vee 0) \wedge (b(\gamma) + \delta) } $, 
	so that $\textsf{f}_n(x^2), n\in\NN$ share a Lipschitz constant $L$, and $G_n(f_n) = G_n(\textsf{f}_n)$ on $E_{1,n}^c$ for all $n\in\NN$. 
	Hence, $ \mathbb{P}(E_{1,n}^c \cap E_{3,n}(f_n, M)) \leq \mathbb{P}( E_{3,n}(\textsf{f}_n, M))$.
	Meanwhile, we have
	\[
	\PP{| p ( F_n(\textsf{f}_n) - \EE{ F_n(\textsf{f}_n) } )| > M }
	\leq 2 \exp (-  M^2 / (  2L^2 ))
	\]
	for $M >0, n\in \NN$ from the Corollary 1.8 of \cite{guionnet2000concentration}(or Lemma A.4 of \cite{paul2007asymptotics}).
	For all $p \geq 1$, from the identity
	$ \EE{|X|^p} = p \int_{0}^{\infty} y^{p-1} \mathbb{P} \p{|X| > y} dy$,
	it follows that $\{ p (F_n(\textsf{f}) - \EE{ F_n(\textsf{f}) } ) \}_{n \in \NN}$ is bounded in $L_p$.
	i.e. is uniformly integrable and thus tight.
	But we assume that $\{ G_n(f_n) \}_{n \in \NN} = \{ G_n(\textsf{f}_n) \}_{n \in \NN} = \{ p \p{ F_n(\textsf{f}_n) - F_{\gamma_n}(\textsf{f}_n) } \}_{n \in \NN}$ is also tight, hence by triangle inequality 
  $ M(\{f_n\}_{n\in\NN}) 
	 = \sup_{n\in\NN}{| p \p{ \EE{ F_n(\textsf{f}_n) } - F_{\gamma_n}(\textsf{f}_n) } |}$
	is finite. Consequently, for $M > 2M(\{f_n\}_{n\in\NN})$,
	\begin{align*}
	\mathbb{P}(  E_{4,n} (\textsf{f}_n, M) ) 
	&\leq \mathbb{P}\p{|  p \p{ F_n(\textsf{f}_n) - \EE{ F_n(\textsf{f}_n) } } | > M - M(\{f_n\}_{n\in\NN})} \\
	&\leq \mathbb{P}\p{|  p \p{ F_n(\textsf{f}_n) - \EE{ F_n(\textsf{f}_n) } } | > M/2}
	\leq 2 \exp (- M^2 / (8L^2)),
	\end{align*}
	as desired.
	
	\noindent{\bf S2.3. \Cref{G_n_lemma}} 
	
        First, note that in view of the Vitali-Porter and Weierstrass
        theorems(e.g. \citet[Ch.~1.4, 2.4]{schiff2013normal}), there
        exists a neighborhood $\Omega_1$ of $I$ with compact closure
        $\bar{\Omega}_1 \subset \Omega$ such that $f_n$ and $f_n'$
        converge uniformly to $f$ and $f'$ respectively on
        $\bar{\Omega}_1$ and so in particular $\{ f_n \}_{n\in\NN}$
        and $\{ f_n' \}_{n\in\NN}$ are each uniformly bounded on
        $\bar{\Omega}_1$.

The truncation and centralization step runs parallel to
\citet[pp.~559-560]{bai2004clt}, [BS] below.  Let $\tilde{G}_n(\cdot)$
denote the analog of $G_n(\cdot)$ with matrix $B_n$ -- which
does not depend on $f, f_n$ -- replaced by $\tilde{B}_n$.
Then the argument there
shows that $\tilde{G}_n(f) - G_n(f)$ and
$\tilde{G}_n(f_n) - G_n(f_n) \stackrel{p}{\to} 0$ because
$f, \{f_n' \}_{n\in\NN}$ are uniformly bounded on $\bar{\Omega}_1$.
Therefore, it suffices to consider when $G_n(\cdot)$ denotes the
centered linear spectral statistic based on the truncated and centered
variables.

Now we argue as on [BS] p.563.
Let $M_n(z)$ be the normalized
Stieltjes transform difference and $\hat{M}_n(z)$ be its modification
on $\mathcal{C}$ as defined on [BS, p.561] 
-- none of these depend
on $f, f_n$. For all large $n$, we have
\begin{equation*}
G_n(f_n) - G_n(f) = - \frac{1}{2 \pi i} \int [f_n(z) - f(z)] M_n(z)
dz
\end{equation*}
almost surely. 
In addition, by arguing as shown on [BS] p. 563,
\begin{equation*}
\int [f_n(z) - f(z)][ M_n(z) - \hat{M}_n(z)]  dz \stackrel{p}{\to} 0
\end{equation*}
as $n \to \infty$ because $f_n$ are uniformly bounded on
$\bar{\Omega}_1$ which contains the contour of integration.

Finally,
\begin{equation*}
\left| \int [f_n(z) - f(z)] \hat{M}_n(z)  dz \right|
\leq \| f_n - f\|_\infty \int |\hat{M}_n(z)| dz
\stackrel{p}{\to} 0,
\end{equation*}
since $f_n \to f$ uniformly on $\bar{\Omega}_1$ and, crucially,
$\{ \hat{M}_n(\cdot) \}$ is a tight sequence on $C(\mathcal{C}, \RR^2)$ as
shown in Lemma 1 of [BS], 
and hence so is $\int |\hat{M}_n(z)| dz$.
	
	\noindent{\bf S2.4. \Cref{Cor_limiting_moment}}
	
	Let $k\in\NN$. 
	From the proof of \Cref{TailBound_LSS}, $\{ (G_n(\textsf{f}_n))^k\} _{n\in\NN}$ is uniformly integrable by 
	$ \EE{|X|^p} = p \int_{0}^{\infty} y^{p-1} \mathbb{P} \p{|X| > y} dy, p \geq 1$ again. 
	Also, from \Cref{G_n_lemma} and continuous mapping theorem, 
	$ (G_n(\textsf{f}_n))^k \overset{d}{\to} (N(\mu(f), \sigma^2(f))^k$.
	Therefore, by Theorem 2.20 of \cite{van2000asymptotic}, a combination of Skorokhod representation theorem and Vitali's convergence theorem, we obtain
	$\lim_{n\rightarrow\infty}\EE{(G_n(\textsf{f}_n))^k} = \tau_k(f).$
	Also, $\Big| \EE{I(E_{n}^c) (G_n(\textsf{f}_n))^k} \Big| \leq \PP{E_{n}^c} \EE{(G_n(\textsf{f}_n))^{2k}} = o(1)$ by Cauchy and the assumption $\lim_{n\rightarrow \infty} \PP{E_n} = 1 $,
	hence it follows from another assumption $E_n \subset E_{1, n}^c$ and $G_n(f_n) = G_n(\textsf{f}_n)$ on $E_{1,n}^c$ that
	\[
	\lim_{n\rightarrow\infty}\EE{I(E_{n}) (G_n(f_n))^k}
	= \lim_{n\rightarrow\infty} \p{ \EE{(G_n(\textsf{f}_n))^k} - \EE{I(E_{n}^c) (G_n(\textsf{f}_n))^k} }
	= \tau_k(f),
	\]
	as desired.
\par

\vskip 14pt
\noindent {\large\bf Acknowledgements}

We thank Zhou Fan for providing the main idea of the proofs of \Cref{TailBound_LSS} and \Cref{1stEdgeworthSumIndVars}.
Supported by NIH R01 EB001988, NSF DMS 1407813 and a Samsung scholarship. 

\par

\markboth{\hfill{\footnotesize\rm JEHA YANG AND IAIN JOHNSTONE} \hfill}
{\hfill {\footnotesize\rm EDGEWORTH CORRECTION FOR ROY'S STATISTIC IN A SPIKED PCA MODEL} \hfill}

\bibhang=1.7pc
\bibsep=2pt
\fontsize{9}{14pt plus.8pt minus .6pt}\selectfont
\renewcommand\bibname

\markboth{\hfill{\footnotesize\rm JEHA YANG AND IAIN JOHNSTONE} \hfill}
{\hfill {\footnotesize\rm EDGEWORTH CORRECTION FOR ROY'S STATISTIC IN A SPIKED PCA MODEL} \hfill}

\vskip .65cm
\noindent
Department of Statistics, Stanford University, Stanford, CA 94305, U.S.A.
\vskip 2pt
\noindent
E-mail: jeha@stanford.edu
\vskip 2pt

\noindent
Department of Statistics, Stanford University, Stanford, CA 94305, U.S.A.
\vskip 2pt
\noindent
E-mail: imj@stanford.edu

\end{document}